\documentclass[reqno, english,11pt]{smfart}
\RequirePackage{mathrsfs} \let\mathcal\mathscr
\usepackage{a4wide}

\usepackage{color}

\usepackage{lmodern}

\usepackage{tikz}

\usepackage{amssymb,amsmath,amsfonts, amsthm} 

\usepackage{hyperref}
\usepackage{dsfont}
\usepackage{upgreek}
\usepackage{mathrsfs}
\usepackage{mathabx,yfonts}
\usepackage{enumerate}

\usepackage{enumitem}

\usepackage{graphicx}

\let\emptyset\varnothing

\newtheorem{theorem}{Theorem}

\newtheorem{lemma}[theorem]{Lemma}
\newtheorem{proposition}[theorem]{Proposition}
\newtheorem{corollary}[theorem]{Corollary}
\theoremstyle{definition}
\newtheorem{definition}[theorem]{Definition}
\newtheorem{remark}[theorem]{Remark}

\newtheorem*{notation}{Notation}

\numberwithin{theorem}{section}
\numberwithin{equation}{section}

\DeclareSymbolFont{bbold}{U}{bbold}{m}{n}
\DeclareSymbolFontAlphabet{\mathbbold}{bbold}

\renewcommand{\mod}[1]{ \ \left(\textnormal{mod}\ #1\right)}
\newcommand{\md}[1]{  \left(\textnormal{mod}\ #1\right)}
 
\renewcommand{\P}{\mathbb{P}}
\newcommand{\Q}{\mathbb{Q}}

\newcommand{\N}{\mathbb{N}}
\newcommand{\A}{\mathbb{A}}
\newcommand{\R}{\mathbb{R}}

\newcommand{\Z}{\mathbb{Z}}
\newcommand{\C}{\mathbb{C}}

\renewcommand{\l}{\left}

\renewcommand{\r}{\right}
\renewcommand{\b}{\mathbf}
\renewcommand{\c}{\mathcal} 
\renewcommand{\gcd}{\textrm{gcd}} 
\renewcommand{\leq}{\leqslant}
\renewcommand{\geq}{\geqslant}
\renewcommand{\#}{\sharp}
\renewcommand{\gg}{\ggg}

\renewcommand{\ll}{\lll}

\newcommand{\p}{\mathfrak{p}}

\newcommand{\1}[2]{{\mathbf{1}}_{\upepsilon_{#1}}(#2)}

\newcommand{\wbk}{\widetilde{\mathbf{k}}}

\newcommand{\bfa}{\mathbf{a}}

\newcommand{\bfd}{\mathbf{d}}
\newcommand{\bfe}{\mathbf{e}}
\newcommand{\bff}{\mathbf{f}}
\newcommand{\bfg}{\mathbf{g}}
\newcommand{\bfh}{\mathbf{h}}
\newcommand{\bfj}{\mathbf{j}}
\newcommand{\bfl}{\mathbf{l}}
\newcommand{\bfx}{\mathbf{x}}
\newcommand{\bfy}{\mathbf{y}}
\newcommand{\bfz}{\mathbf{z}}
\newcommand{\bfk}{\mathbf{k}}
\newcommand{\bfs}{\mathbf{s}}
\newcommand{\bfu}{\mathbf{u}}

\newcommand{\bfw}{\mathbf{w}}
\newcommand{\bfP}{\mathbf{P}}

\newcommand{\calE}{\mathcal{E}}
\newcommand{\calI}{\mathcal{I}}

\newcommand{\calP}{\mathcal{P}}
\newcommand{\calY}{\mathcal{Y}}
\newcommand{\grm}{\mathfrak{m}}

\newcommand{\grM}{\mathfrak{M}}
\newcommand{\grS}{\mathfrak{S}}
\newcommand{\alp}{{\alpha}}
\newcommand{\bet}{{\beta}}
\newcommand{\tet}{{\theta}}
\newcommand{\del}{{\delta}}
\newcommand{\lam}{{\lambda}}
\newcommand{\eps}{{\epsilon}}
\newcommand{\vareps}{{\varepsilon}}
\newcommand{\sig}{{\sigma}}
\newcommand{\ome}{{\omega}}
\newcommand{\gam}{{\gamma}}
\newcommand{\Del}{{\Delta}}
\newcommand{\Sig}{{\Sigma}}

\newcommand{\Lam}{{\Lambda}}
\newcommand{\bfgam}{{\boldsymbol \gam}}
\newcommand{\bfalp}{{\boldsymbol \alp}}
\newcommand{\bfbet}{{\boldsymbol \bet}}
\def\d{{\,{\rm d}}}
\newcommand{\rank}{\textrm{rank}} 

\newcommand{\supp}{\textrm{supp}}
\newcommand{\meas}{\textrm{meas}}

\DeclareMathOperator*{\Osum}{\sum{}^*}
\DeclareMathOperator*{\Oprod}{\prod{}^\flat}
\DeclareMathOperator*{\OOprod}{\prod{}^{\flat\flat}}
\begin{document}
\begin{abstract}
We study almost prime solutions
of systems of Diophantine equations in 
the Birch setting.
Previous work shows that there exist integer solutions of size $B$
with each component having 
no prime divisors below $B^{1/u}$, where  
$u=c_0n^{3/2},$
$n$ is the number of variables and $c_0$ is a constant depending on the degree and the number of equations.
We improve the polynomial
growth 
$n^{3/2}$
to the logarithmic
$(\log n) (\log \log n)^{-1}$. 
Our main new ingredients are
the generalisation of
the Br\"{u}dern--Fouvry
vector sieve
in any dimension
and the 
incorporation of smooth weights 
into the Davenport--Birch version of the circle method.
\end{abstract}

\date{\today}
\title[
Sarnak's saturation problem for complete intersections
]
{
Sarnak's saturation problem for complete intersections
}

\author{D. Schindler}
\address{
Universiteit Utrecht\\
Mathematisch Instituut\\
Budapestlaan 6\\
Utrecht\\
3584 CD \\
Netherlands
}
\email{d.schindler@uu.nl}

\author{E. Sofos}
\address{
Max Planck Institute for Mathematics\\
Vivatsgasse 7\\
Bonn\\ 
53111\\
Germany 
}
\email{sofos@mpim-bonn.mpg.de}


\maketitle

\setcounter{tocdepth}{1}
\tableofcontents

\section{Introduction}
\label{intro}

Let $f_1,\ldots, f_R\in \Z[x_1,\ldots, x_n]$ be forms
of degree $d$ and write $\bff=(f_1,\ldots, f_R)$. 
We consider the affine
variety defined by
\begin{equation}\label{eqn0}
V_\bff
: f_i(x_1,\ldots, x_n)=0,\quad 1\leq i\leq R.
\end{equation} 
We are interested in 
Sarnak's saturation problem,
that is to find 
a
Zariski-dense set 
of 
integer zeros $(x_1,\ldots,x_n) \in V_\bff(\Z)$
where each $x_i$ is either a prime or has a small number of prime divisors.
Recent work of Cook and Magyar \cite{CM} is concerned with finding prime solutions to the Diophantine system
$\bff(\bfx)=\bfs$ for $\bfs\in \Z^R$, i.e. solutions for which every variable $x_i$ is a prime number. 
They succeed in establishing a local to 
global principle, including an asymptotic formula, via
the circle method if the Birch rank $\mathfrak{B}(\b{f})$,
that will be defined at the beginning of \S\ref{s:genebi}),
satisfies $\mathfrak{B}(\b{f})\geq \chi(R,d)$
for some 
function 
$\chi(R,d)$
which only depends on the degree $d$ and the number of polynomials $R$. 
However, the value of $\chi(R,d)$, as it would result from the current proof in \cite{CM}, is expected to be tower exponential
in $d$ and $R$. 
For systems of quadratic forms one has $$\chi(R,2)\leq 2^{2^{CR^2}}.$$ 
For more general systems we do not have any explicit upper bounds on this function.
\par
It is therefore
natural to ask whether one can find an explicit condition
which ensures the existence of
a Zariski dense set of integer solutions
with all coordinates being almost prime;
this is usually referred to as Sarnak's saturation problem.
Let $\Omega(m)$ denote the number of prime factors of $m$ counted with multiplicity.
Almost primes have zero density in the integers
owing to the generalised prime number theorem:
for each fixed integer
$k \geq 1$ 
we have 
\[
\frac{1}{x}
\#\Big\{m \in \mathbb{N} \cap [1,x]: \Omega(m)\leq k
\Big\}
\sim
\frac{(\log \log x)^{k-1}}{(k-1)!
\log x},
\text{ as } x \to\infty
.\]
The fact that
one seeks solutions in 
thin subsets of integers
places
problems of this type in a higher level of difficulty than studying the number of all integer solutions 
in expanding regions. 
Yamagishi~\cite{arXiv:1709.03605} 
showed the existence of infinitely many 
integer solutions in the case $R=1$ and for large $n$,
with
every solution 
having exactly $2$ prime factors.
This corresponds to taking $k=2$ in the asymptotic above.

In this paper we are interested  in a harder question than that of almost primes, 
namely in finding solutions within sets that have asymptotically zero density 
compared to the set of almost primes. 
Let $P^-(m)$ denote the least prime divisor of a positive integer $m\neq 1$ 
and define $P^-(1):=1$.
Integers $m$ with $P^-(m)\geq m^{1/u}$ for some $u>1$
are almost primes,
however their density is \textit{arbitrarily smaller} in comparison. 
Indeed, by Buchstab's theorem~\cite{buch}  
one has the following 
for all fixed 
$k\in \N_{\geq 2}$ and $u \in \R_{>1}$,
\[
\frac
{\#\big\{m \in \mathbb{N} \cap [1,x]: P^-(m)\geq x^{1/u}\big\}}
{\#\big\{m \in \mathbb{N} \cap [1,x]: \Omega(m)\leq k\big\}}
\sim
\frac
{(k-1)!u w(u)}
{(\log \log x)^{k-1}}
\ll_{k,u} \frac{1}
{(\log \log x)^{k-1}},
\text{ as } x \to\infty
,\]
where $w(u)$ is the Buchstab function.
Progress on the saturation problem within this thinner set of solutions
was recently made by Magyar and Titichetrakun \cite{magtit}. 
They managed to treat systems of equations where the number of variables is the same as in Birch's work \cite{Bir62}, i.e. assuming that the Birch rank exceeds 
$R(R+1)(d-1)2^{d-1}$. 
They proved
lower bounds of the correct order of magnitude
regarding the number of integer solutions 
with each coordinate $x_i$ satisfying
$P^-(|x_i|)\geq |x_i|^{1/u}$, where $u$ is any constant in the range
\begin{equation}
\label{eq:bmagt}
u\geq 
2^8n^{3/2}d(d+1)R^2(R+1)(R+2)
.\end{equation}
The ultimate goal of showing that 
all variables $x_i$ can simultaneously be 
prime 
corresponds to the value
$u>2-\epsilon$ for some $\epsilon>0$,
hence any result decreasing the admissible value for $u$ in~\eqref{eq:bmagt}
is an equivalent reformulation of progress towards this goal.
Our aim in this paper 
is
to decrease the admissible value for $u$
when the degree and the number of equations is
fixed
so as to have 
at most logarithmic growth in terms of $n$
rather than 
polynomial.
\subsection{Summary of our results}
\label{s:summ}
In order to prove quantitative or qualitative
results for the system of equations (\ref{eqn0}) one typically needs $n$ to be 
sufficiently large in terms of $d$ and $R$ and the singular locus of $V_{\bff}$. Thus, for example, the Hasse principle is known
for non-singular cubic hypersurfaces when $n\geq 9$
(Hooley~\cite{MR936992}),
for non-singular quartics when $n\geq 40$ (Hanselmann~\cite{hanse})
and for non-singular quintics in at least $n\geq 101$ variables (Browning and Prendiville~\cite{brpre}).
One may expect that the dependence of $\eqref{eq:bmagt}$ 
on $n$ should decrease when $n$ increases;
we are not able to provide a bound that is independent of $n$ 
but we shall provide a bound that depends logarithmically on $n$
rather than polynomially.
For this 
we shall use the vector sieve of Br\"{u}dern
and Fouvry to show that 
for fixed $d,R$ one can improve~\eqref{eq:bmagt}
to 
\[
u\gg \frac{\log n}{\log \log n}
.\]
This constitutes a major improvement
over~\eqref{eq:bmagt}
and it applies to almost all situations in the Birch setting,
see Theorem~\ref{thm:mainvector}.
This is the main result in this paper.

As an additional result we shall provide an improvement in all situations in the Birch setting,
however,
this will not be of logarithmic nature.
Namely, using the Rosser--Iwaniec sieve 
we shall prove that one can take 
$u\gg n$ 
in all of the remaining cases, see Theorem~\ref{thm:rosiwa}, 
while, in some situations covered by Theorem~\ref{thm:rosiwa}
but not by Theorem~\ref{thm:mainvector}
we shall show via the weighted sieve 
that there are many integer zeros 
$(x_1,\ldots,x_n)$
where
the total number of prime factors of 
$|x_1\cdots x_n|$
is
$\ll n \log n$, while at the same time every prime factor of each $|x_i|$ is at least $|x_i|^\alpha$
for some $0<\alpha<1$ independent of $\b{x}$, see
Theorem~\ref{thm:weighted}.

\subsection{The vector sieve in arbitrary dimension}
\label{s:vectt}
The \textit{vector sieve}
was brought into light
by Br{\"u}dern
and Fouvry~\cite{BF}
to show that 
for all sufficiently large 
positive integers $N$ satisfying
$N\equiv 4 \mod{24}$
the Lagrange equation
\[
N=x_1^2+x_2^2+x_3^2+x_4^2
\]
has many solutions $\b{x} \in \N^4$ with each $x_i$ 
being indivisible by any prime
of size at most $N^{1/u}$ with 
$u\geq 68.86$.
Problems of type Waring--Goldbach become less hard
the more variables are available
and the expectation is that
one can take each $x_i$ to be a prime for  $N$ as above-
this is still open while the case of representations by $5$ squares of primes
was settled by Hua~\cite{MR3363459}.
The vector sieve was later used to make improvements 
on the admissible value for $u$
in Lagrange's equation
by
Heath-Brown and Tolev~\cite{MR1979185},
Tolev~\cite{MR2075638}
and
Cai~\cite{MR2755472},
as well as in other sieving
problems
(\cite{MR1401427},~\cite{MR2143730},~\cite{MR3531240}).  

The main idea of the vector sieve is to use a combinatorial 
inequality that replaces the usual lower bound sieve 
by a linear combination of products of sieving functions each of dimension $1$,
one of the advantages being an improvement over the admissible value for $u$.
There are other applications of the vector sieve in the literature
but to our knowledge it has not been applied for sieves of 
arbitrarily large sieve dimension (the reader is referred to the book of 
Friedlander and Iwaniec~\cite{MR2647984}
for the terminology). 

Let us now proceed to the statement of our main theorem.
Denoting the $p$-adic units 
by $\Z_p^\times$ 
we will always make the assumption that  
\begin{equation}
\label{cond:units}
\b{f}=\b{0} \text{ has non-singular solutions in } (0,1)^n \text{ and in } (\Z_p^\times)^n
\text{ for every prime } p.
\end{equation}
We shall define the quantity $K=K(\bff)$ in~\eqref{def:birchrank}
using the notion of
the
Birch rank
$\mathfrak{B}(\b{f})$.
Let  
\begin{equation}
\label{def:theY} 
\Upsilon:=
\frac{d\mathfrak{B}(\bff)}{(d-1)2^{d-1}}
\Big(d-\frac{1}{R}\Big)
+R 
,\end{equation}
as well as 
\begin{equation}
\label{def:the} 
\theta':=
\min
\Big\{
\frac{1}{\rho},
\frac{\epsilon_{1,1}-dR}{\epsilon_{1,2}+\epsilon_{1,3}},
\frac{\epsilon_{2,1}-dR}{\epsilon_{2,2}+\epsilon_{2,3}},
\frac{\epsilon_{3,1}-dR}{\epsilon_{3,2}+\epsilon_{3,3}}
\Big\}
,\end{equation}
where
\begin{equation}
\label{def:rrr}
\rho:=4R(R+1)d\l(1+\frac{d}{2R(d-1)+1}+\frac{3Rd}{3R(d-1)+1}\r)
\end{equation}
and  the vectors $(\epsilon_{i,1},\epsilon_{i,2},\epsilon_{i,3})$
are defined
as
the 
columns
of the following
matrix,
\begin{equation}
\label{def:epss}
\boldsymbol{\epsilon}:=
\begin{bmatrix} 
{Rd+1/2} & K& \frac{(K-R(R+1)(d-1))}{4R(R+1)d}+Rd
 \\
Rd+1/2 & K & \frac{d(K-R^2(d-1))}{2R(d-1)}+R+K+\frac{2dK}{d-1}-Rd \\  
0  & 0 & \max\big\{0,\frac{K-R(R+1)(d-1)}{4R(R+1)d}-R-K+Rd\big\}
\end{bmatrix}
 .\end{equation} 
\newline
Here follows the main result of our paper.
\begin{theorem}
\label{thm:mainvector}
There exists a positive absolute constant $c_0$ such that whenever 
the forms $f_1,\ldots,f_R\in \Z[x_1,\ldots,x_n]$ of degree $d\geq 2$
satisfy~\eqref{cond:units}
and
\[
\mathfrak{B}(\b{f})
>
\max\big\{
2^{d-1}(d-1)
R(R+1),
2^{d-1}
(d-1)R^2
+(R+1)
(\Upsilon+1),
2^{d-1}
(d^2-1)R^2
\big\}
,\]
then we have for all large enough $B\geq 1$, 
\[
\#
\Big\{
\b{x} \in ((0,B]\cap\N)^n:\b{f}(\bfx)=\b{0},
P^-(x_1\cdots x_n)>
B^{
\theta'
\frac{\log \log n}
{c_0\log n}
}
\Big\}
\gg
\frac{B^{n-Rd}}{(\log B)^n}
,\]
where the constant $\theta'$ satisfies $\theta'\gg_{d,R} 1$.
\end{theorem}
This provides a
lower bound 
$\log P^-(x_1\cdots x_n)/\log B$
in terms of $n$ 
that vanishes logarithmically slow as $n\to+\infty$, which constitutes 
a large improvement over 
the
previously
best known result that gave
a polynomial decay~\cite{magtit}.
The proof of Theorem~\ref{thm:mainvector}
will be given in \S\ref{s:mainthm}.
A crucial input for the sieving arguments 
will be a general 
version of Birch's theorem that we
shall prove in~\S\ref{s:genebi},
see Theorem~\ref{thm:master2}. Note that a similar result for one quadratic form is proved in work of Browning and Loughran \cite[Theorem 4.1]{BroLou17}, whereas our result aims at 
general complete intersections.  
More importantly, 
Theorem~\ref{thm:master2}
allows situations where 
congruence conditions 
are imposed to every integer coordinate
with a different moduli for every coordinate,
while in their result one is only allowed to consider the same moduli for every coordinate.
This extra feature will be of central importance for the vector sieve.

An inspection of the argument at the end of \S\ref{s:mainthm}
shows that 
we
can take $c_0=3$ 
in
Theorem~\ref{thm:mainvector}
when the number of variables $n$
is 
sufficiently large.
For $s\in \mathbb{R}_{>2}$ let $0<f(s)\leq 1\leq F(s)$ 
be the sieve functions associated to the linear Rosser--Iwaniec sieve, defined for example in~\cite{MR581917}, which
satisfy $F(s),f(s)=1+O(s^{-s})$.
One can improve the lower bound for 
$\log P^-(x_1\cdots x_n)/\log B$
given by Theorem~\ref{thm:mainvector}
by
replacing
the term
$
\frac{c_0\log n}{\log \log n}
$
by
any value $s>2$
that satisfies 
\[F(s)^n<\Big(1+\frac{1}{n-1}\Big)f(s).\]

A special case of Theorem~\ref{thm:mainvector} is 
the case of non-singular hypersurfaces.
\begin{corollary}
\label{cor:mainvector}
There exists a
positive
absolute constant $c_1$ such that 
whenever
$f$
is an integer non-singular
form 
of degree $d\geq 5$
 in more than $2^{d-1}(d^2-1)$
variables
that fulfils
~\eqref{cond:units}
then the following estimate
holds for all large enough $B\geq 1$, 
\[
\#
\Big\{
\b{x} \in ((0,B]\cap\N)^n:f(\bfx)=0,
P^-(x_1\cdots x_n)>
B^{\frac
{c_1 \log \log n}
{d \log n}
}
\Big\}
\gg
\frac{B^{n-d}}{(\log B)^n}
.\] 
\end{corollary}
Our results require a few more variables than in the Birch setting, 
which for non-singular hypersurfaces
requires
$n>2^{d}(d-1)$.
The reason for this is rooted to the way that the vector sieve works:
in introducing $n$ linear sieving functions  
in place of a single $n$-dimensional lower bound sieve 
the technique requires that we have a good control on 
the independency of the events that a large prime $p$
divides several coordinates of an integer zero,
this is related to the function $\delta$ that will be studied in~\S\ref{s:localden}.
The Birch assumption $$\mathfrak{B}(\bff)>2^{d-1}(d-1)R(R+1)$$
does not always allow a good bound for $\delta$, however
a slightly stronger geometric 
assumption will be shown to be sufficient via a version of 
Weyl's inequality that is uniform in the coefficients
of the underlying polynomials.
It
must
be noted that the 
work of 
Yamagishi~\cite[Th.1.3]{arXiv:1709.03605} 
only
applies to 
smooth
hypersurfaces 
in  
$n>8^d (4d-2)$ variables,
which ought to be compared 
with the assumption 
$n>2^{d-1}(d^2-1)$
 of Corollary~\ref{cor:mainvector}

\subsection{Applications to the saturation problem}

One further advantage of Theorem~\ref{thm:master2}
is that it allows the use of any smooth weight with compact support.
We can therefore establish
a version of Theorem~\ref{thm:mainvector}
where one counts solutions near an arbitrary non-singular point in $V_\bff(\R)$.
This allows to settle Sarnak's problem for the complete intersections under consideration.
To phrase our result 
we first need the following definition.
Each
 $x\in \P^{n-1}(\Q)$ can be written uniquely
up to sign
in the form $x=[\pm \b{x}]$, where $\b{x}=(x_1,\ldots,x_{n})\in \Z^n$ and $\gcd(x_1,\ldots,x_{n})=1$.
We can then define the function $\c{L}:\P^{n-1}(\Q)\to \R_{\geq 0}$
through
\[
\c{L}(x):=\max_{\substack{1\leq i \leq n\\ x_i\neq 0}}
\max 
\bigg\{\frac{\log |x_i|}{\log p}:
p \text{ is a prime  dividing }  x_i
\bigg\}
.\]
Thus
$\c{L}(x)\leq u$ holds
for some
$x=[\pm
(x_1,\ldots,x_n)]\in \P^{n-1}(\Q)$ and $u\in \R_{\geq 0}$
if and only if 
\[
x_i\neq 0
\Rightarrow 
P^{-}(|x_i|)
\geq 
|x_i|^{1/u}
.\]
\begin{definition}
[Level of saturation]
\label{def:roughness}
Assume that 
$X\subset \P^{n-1}$ is a variety defined over $\Q$.
The 
level of saturation
of $X$ 
is the infimum of all real non-negative
numbers $u$ 
such that  
\[
\{x\in X(\Q):\c{L}(x)\leq u
\}
\]
is Zariski dense in $X$.
\end{definition}
Note that in this definition 
the level of saturation is allowed to be infinite,
for example if $X(\Q)$ is not Zariski dense.
Recalling the definition of the number of prime divisors 
$\Omega_{\P^{n-1}(\Q)}(x)$
of a rational point $x\in \P^{n-1}(\Q)$
in the paragraph before~\cite[Def.1.1]{yuchao}, we observe that if $\prod_{i}x_i\neq 0$
then 
\[\Omega_{\P^{n-1}(\Q)}(x)
\leq n
\c{L}(x)
.\]
Therefore,
according to~\cite[Def.1.1]{yuchao},
if $X$ has a finite level of saturation 
then it has a finite saturation number.
Therefore
one could perceive
Definition~\ref{def:roughness}
as 
a refinement
of the standard
notion of saturation.

\begin{theorem}
\label{thm:levelsaturation}
There exists a positive absolute constant $c_0$ such that whenever 
the forms $f_1,\ldots,f_R\in \Z[x_1,\ldots,x_n]$ of degree $d\geq 2$
satisfy~\eqref{cond:units}
and
\[
\mathfrak{B}(\b{f})
>
\max\big\{
2^{d-1}(d-1)
R(R+1),
2^{d-1}
(d-1)R^2
+(R+1)
(\Upsilon+1),
2^{d-1}
(d^2-1)R^2
\big\}
\]
and 
the complete intersection in $\P^{n-1}$ that is defined through
\[
V_\b{f}:
\
f_1=f_2=\cdots=f_R=0\]
is geometrically irreducible
then $V_\b{f}$
has finite 
level of saturation.
In addition, the 
level of saturation
is at most  
\[ 
\frac{c_0\log n}
{
\theta'
 \log \log n}
,\]
where the constant $\theta'$ satisfies $\theta'\gg_{d,R} 1$.
\end{theorem}

\subsection{Results via the Rosser--Iwaniec sieve}
We next provide an almost prime result that covers all situations in the Birch setting,
thus completing the treatment of the cases not covered by
Theorem~\ref{thm:mainvector}.
This will provide a lower bound for
$\log P^-(x_1\cdots x_n)/\log B$
that is worse than the one in Theorem~\ref{thm:mainvector}
but still better than~\eqref{eq:bmagt};
this is due to the strength of the level of distribution
result implied by Theorem~\ref{thm:master2}. 
\begin{theorem}
\label{thm:rosiwa}   
For any forms $f_1,\ldots,f_R\in \Z[x_1,\ldots,x_n]$ of degree $d\geq 2$
satisfying~\eqref{cond:units}
and $K>R(R+1)(d-1)$
we have for all large enough $B\geq 1$, 
\[
\#
\Big\{
\b{x} \in ((0,B]\cap\N)^n:\b{f}(\bfx)=\b{0},
P^-(x_1\cdots x_n)>
B^{ \frac{\theta'}{3.75 n}}
\Big\}
\gg
\frac{B^{n-Rd}}{(\log B)^n}
,\]
where $\theta'$ is given in~\eqref{def:the} and satisfies $\theta'\gg_{d,R} 1$.
\end{theorem}
\subsection{Results via the weighted sieve}
\label{s:we8}
Theorem~\ref{thm:rosiwa}
supplies a 
polynomially
fast convergence to zero
for
$\log P^-(x_1\cdots x_n)/\log B$ with respect to $n$.
This is slightly undesired, thus
we shall provide
a complementary result 
that furnishes many integer zeros satisfying a bound of similar quality for
$\log P^-(x_1\cdots x_n)/\log B$ with the additional desired property
that 
$x_1\cdots x_n$ has few prime factors. 
This will be implemented via the weighted sieve.
We choose
to include this result here because along the proof
we shall
provide a potentially
useful
reformulation of the weighted sieve given in the book of
Diamond and Halberstam~\cite{MR2458547}.
This reformulation allows 
the
incorporation of further weights
and will be given in Theorem~\ref{thm:weighted0}.

Define
\begin{equation}
\label{eq:epikoskafes}
u'':=(n-Rd) 
\max\Bigg\{\frac{
(2\epsilon_{i,2}-1)}{\epsilon_{i,1}-Rd}:1\leq i \leq 3
\Bigg\},\end{equation}
\begin{equation}
\label{eq:epikoskafes2}
\widehat{u}:=
\max\Big\{u'',
1/\theta',
2(n-Rd)
\rho\Big\},
\
\widehat{v}
:=
\frac{n c_n-1}{\theta'-1/\widehat{u}}
,\end{equation}
where
$c_n$ is a sequence that satisfies
$\lim_{n\to+\infty}c_n=2.44\ldots$.
We furthermore let 
\begin{equation}
\label{eq:epikoskafes3}
r_0:=
\frac{n\widehat{u}}{n-Rd}
-1
+
n\Big(1+\frac{\widehat{u}}{\widehat{v}}c_n\Big)
\log \frac{\widehat{v}}{\widehat{u}}
-n
\Big(1-\frac{\widehat{u}}{\widehat{v}}\Big)
.\end{equation}
\begin{theorem}
\label{thm:weighted}
For any forms $f_1,\ldots,f_R\in \Z[x_1,\ldots,x_n]$ of degree $d\geq 2$
satisfying~\eqref{cond:units}
and   
$\mathfrak{B}(\b{f})
>
\max\{(d-1)R(R+1)2^{d-1},
(d^2-1)R2^{d-1}
,
(d-1)R^2
2^{d-1}
+2(R+1)
\}
$
we have for all $r_1>r_0$
and all large enough $B\geq 1$, 
\[
\#
\Big\{
\b{x} \in ((0,B]\cap\N)^n:\b{f}(\bfx)=\b{0},
P^-(x_1\cdots x_n)
>B^{1/\widehat{v}},
\Omega(x_1\cdots x_n) 
\leq   r_1
\Big\}
\gg
\frac{B^{n-Rd}}{(\log B)^n}
,\]
where
$\widehat{v}=O_{d,R}(n)$ and
$
r_0=O_{d,R}(n \log  n)
$.
\end{theorem}
A simple 
consequence of Theorem~\ref{thm:mainvector}
is that it provides many integer zeros $\bfx$ with 
\[
\Omega(x_1\cdots x_n) \ll \frac{n \log n}{\log \log n}
,\]
which constitutes an asymptotic saving compared to the estimate 
\[
\Omega(x_1\cdots x_n) \ll n \log n 
\]
supplied by Theorem~\ref{thm:weighted}.
This is surely
surprising to those familiar with 
the weighted sieve and its applications to higher dimensional sieve problems.
The reason that the vector sieve gives a better
saturation number here is the strong level of distribution supplied by 
Theorem~\ref{thm:master2}, which is a result  of using smooth weights.
Indeed, Theorem~\ref{thm:master2} allows to estimate asymptotically the  
number of 
integer solutions of $\bff(\bfx)=\b{0}$ subject to divisibility conditions
of the form
$k_i|x_i$
for $\bfx$ 
in a region having the shape
$\bfx \in B[-1,1]^n$ 
and vectors $\bfk \in \N^n$ of size 
$|\bfk|\leq B^{1/s}$,
where $s>1$ depends on $d$ and $R$
but not on $n$.
Such a level of distribution is usually
not available 
in other problems related to the weighted sieve.
\begin{notation}
We shall reserve the symbol 
$\nu(m)$ for the counting function of distinct prime factors of a positive integer $m$.
For vectors 
$\b{x} \in \R^n$, $n \in \N$,
we shall reserve the symbols
$|\b{x}|$ and 
$|\b{x}|_1$ 
for the supremum and the $\ell^1$
norm respectively.
For vectors $\bfk,\bfx \in \N^n$ we shall abbreviate the simultaneous conditions
$k_i|x_i$ by $\bfk|\bfx$. Similarly we write $\bfk\leq \bfx$ or $\bfk < \bfx$ or $|\bfk|\leq \bfx$ for the simultaneous conditions $k_i\leq x_i$ (resp. $k_i < x_i$ and $|k_i|\leq x_i$) for $1\leq i\leq n$. 
We shall furthermore find it convenient to introduce the notation
\[
\wbk
:=k_1\cdots k_n
,\]
as well as 
\[
\langle\bfk\bfx \rangle
= (k_1x_1,\ldots, k_nx_n).\] 
For $q \in \N$, $z \in \mathbb{C}$ 
we shall write
\[
e_q(z):=e^{\frac{2 \pi i z}{q}}
\
\text{and}
\
e(z):=e^{2 \pi i z}
.\]
The letter $\epsilon$
will refer to an arbitrarily small positive fixed constant
and to ease the notation we shall not record the dependence of the implied constant in the 
 $\ll$ and $O(\cdot)$ notation.
The letter $w$ will be reserved to denote certain weight functions
that will be considered constant throughout our work, thus 
we shall not record the dependence of the implied constant in the 
 $\ll$ and $O(\cdot)$ notation.
Throughout our work the
forms 
$\bff$ are considered to be constant,
thus
each implied constant in the $\ll$ and $O(\cdot)$ notation 
will depend on the coefficients of $\bff,d,n,z_0$ and $W$,
where the constants $z_0,W$ are functions of $\bff$ whose meaning
will become clear in due course.
Any extra dependencies will be specified by the use of a subscript.
\end{notation}

{\bf Acknowledgements:} We would like to thank Prof. T. D. Browning and Dr. S. Yamagishi for their comments on an earlier version of this paper,
as well as the anonymous referee for numerous helpful comments that have clarified the exposition considerably.
The first author is supported by a NWO grant 016.Veni.173.016.    
\section{A version
of
Birch's theorem with lopsided boxes
and smooth weights
}
\label{s:genebi}
In our applications of sieve methods it will be important to 
be able to count integer zeros of $\bff(\b{x})=\b{0}$ 
such that each integer coordinate $x_i$ is divisible by a fixed integer $k_i\leq |x_i|$.
A change of variables makes clear
that a version of Birch's theorem with lopsided boxes
and with uniformity of the error term in the coefficients of the polynomials is sufficient. 
One can do this without smooth weights however 
the resulting error terms will give a weak level of distribution
for our sieve applications.
We shall instead
use smooth weights 
and as a result we shall later be able 
to take $k_i$ much closer to the size of $x_i$.

We now proceed to describe the version of Birch's theorem that we shall need.
Assume that we are given any finite collection of polynomials 
\[
g_i\in \Z[x_1,\ldots,x_n],
1\leq i \leq R
,\]
denote the homogeneous part of $g_i$ by $g_i^\natural$
and assume that there exists 
$d \in \N_{\geq 2}$
such that 
\[
1\leq i \leq R
\Rightarrow
\deg(g_i^\natural)=d
.\]
The Birch rank, denoted by $\mathfrak{B}(\b{g}^\natural)$,
is defined as the codimension of the affine variety in 
$\mathbb{C}^n$
which is
given by 
$$\mathrm{rank}\Bigg(\left(\frac{\partial g_i^\natural(\bfx)}{\partial x_j}\right)_{\substack{1\leq i\leq R,1\leq j\leq n}}\Bigg)<R.$$
We set
\begin{equation}
\label{def:birchrank}
K:=2^{-(d-1)}\mathfrak{B}(\b{g}^\natural).
\end{equation}
Let us fix any smooth compactly supported weight function $w: \R \rightarrow \R_{\geq 0}$ with the property
$\supp(w) \subset [-2,2]$. For  $\bfP=(P_1,\ldots,P_n) \in (\R_{\geq 1})^n$ 
we denote
$$
\widetilde{\b{P}}
:=\prod_{i=1}^n P_i,
\
P_{\max} 
:= \max_{1\leq i\leq n} P_i \ \text{ and } \ 
P_{\min}
:=\min_{1\leq i\leq n} P_i$$
and fix an element $\bfz\in[-1,1]^n$. Our aim is to find an asymptotic formula for
the counting function
$$N_w(\bfP):= \sum_{\substack{\bfy\in \Z^n\\ \bfg(\bfy)=\mathbf{0}}}
\prod_{i=1}^n w\left(\frac{y_i}{P_i}-z_i\right).$$
Birch's influential work~\cite{Bir62} treated the case where
$w$ is replaced by the characteristic function of a finite interval and
$$K>R(R+1)(d-1),
P_{\min}=P_{\max}.$$
For our 
applications of sieving methods a result that is uniform in the size of each $P_i$
as well as the coefficients of each $g_i$ is required.
For $h \in \mathbb{C}[x_1,\ldots,x_n]$
we denote by $\|\b{h}\|$ the maximum of the absolute values of its coefficients 
and for $h_1,\ldots,h_R \in \mathbb{C}[x_1,\ldots,x_n]$ we let
\[
\|\b{h}\|:=\max\{\|h_i\|:1\leq i \leq R\}
.\] 
\begin{theorem}
\label{thm:master2}
Let $g_i,w,\b{z},P_i$ be as above,
assume that  
$K>R(R+1)(d-1)$
and 
\begin{equation}
\label{cond:master}
\frac{P_{\max}}{P_{\min}}<\Vert \bfg\Vert^{-\frac{1}{2R(d-1)+1}}\Vert \bfg^\natural\Vert^{-\frac{3R}{3R(d-1)+1}} P_{\max}^{\frac{1}{4R(R+1)d}}
.\end{equation} 
Then one has for each $\epsilon>0$,
\begin{align*} 
N_w(\bfP)-\grS J_w
\ll &
\widetilde{\b{P}} \left(P_{\max}/P_{\min}\right)^R P_{\max}^{-Rd-1/2}
+
\widetilde{\b{P}}^{1+\eps} \left(P_{\max}/P_{\min}\right)^{K}P_{\max}^{- K}
\\ 
+& 
\Vert\bfg^\natural\Vert^ 
{\frac{2K}{d-1}-R}
\Vert\bfg\Vert^
{\frac{K-R^2(d-1)}{2R(d-1)}}  
\widetilde{\b{P}}^{1+\eps}
\left(P_{\max}/P_{\min}\right)^{R+K} 
P_{\max}^{-Rd-\frac
{K-R(R+1)(d-1)}
{4R(R+1)d} 
}
,\end{align*}
where the implied constant depends at most on $\epsilon>0$.
Here $\grS$ and $J_\ome$ are the usual circle method singular series and singular integral and are defined in~\eqref{def:convser}
and~\eqref{def:convinte} respectively.
\end{theorem} 

Our sole
aim in this section is to establish Theorem~\ref{cond:master}.
All implied constants may depend on $n,R,d$ but not on the coefficients of the polynomials
$g_i(\bfy)$, $1\leq i\leq n$. We start by
introducing the exponential sum
$$S_w(\bfalp):= \sum_{\bfy\in \Z^n}\prod_{i=1}^n w\left(\frac{y_i}{P_i}-z_i\right)e(\bfalp \cdot \bfg (\bfy)),$$
where we use the vector notation $\bfalp\cdot\bfg(\bfy)=\sum_{i=1}^R \alp_i g_i(\bfy)$. By orthogonality we now have
$$N_w(\bfP)=\int_{[0,1]^R} S_w(\bfalp)\d\bfalp.$$
We shall follow Birch's approach \cite{Bir62} to approximate $N_w(\bfP)$. Our first step is to produce a Weyl type inequality for $S_w(\bfalp)$. 
Recall that $g_i^\natural(\bfy)$ are homogeneous polynomials of degree $d$, which can be written as
$$g_i^\natural(\bfy)= d! \sum_{1\leq j_1,\ldots, j_d\leq n}g_{j_1,\ldots, j_d}^{(i)} y_{j_1}\ldots y_{j_d},$$
with symmetric coefficients $g_{j_1,\ldots, j_d}$ (i.e. such that $g_{j_1,\ldots, j_d}=g_{\sig(j_1),\ldots, \sig(j_d)}$ for a permutation $\sig$ of the indices). We associate its multilinear
forms
$$\Phi_i(\bfy^{(1)},\ldots, \bfy^{(d)}) = d! \sum_{1\leq j_1,\ldots, j_d\leq n} g_{j_1,\ldots, j_d}^{(i)}y_{j_1}^{(1)}\ldots y_{j_d}^{(d)},$$
and set
$$\Phi (\bfy^{(1)},\ldots, \bfy^{(d)}):=\sum_{i=1}^R \alp_i \Phi_i(\bfy^{(1)},\ldots, \bfy^{(d)}).$$

\begin{lemma}
With the notation above we have
$$\frac{|S_w(\bfalp)|^{2^{d-1}}}{\widetilde{\b{P}}^{2^{d-1}}}\ll
\widetilde{\b{P}}^{-d}
\sum_{-2\bfP < \bfh^{(1)} < 2\bfP}\ldots \sum_{-2\bfP< \bfh^{(d-1)}< 2\bfP} \prod_{i=1}^n \min\left\{P_i, \Vert \Phi (\bfh^{(1)},\ldots, \bfh^{(d-1)},\bfe^{(i)})\Vert^{-1}\right\}$$

\end{lemma}

\begin{proof}
For $w(x)$ a weight function and $h\in \R$ we introduce the notation
$$w_h(x)=w(x+h)w(x).$$
Moreover, for $h_1,\ldots, h_m\in \R$, we iteratively define
$$w_{h_1,\ldots, h_m}= w_{h_1,\ldots, h_{m-1}}(x+h_m)w_{h_1,\ldots, h_{m-1}}(x).$$
The same Weyl differencing process as in the proof of Lemma 3.3 (in particular equation (3.5)) in \cite{brpre} or in Lemma 2.1 in \cite{Bir62} leads to 
$$\frac{|S_w(\bfalp)|^{2^{d-1}}}{\widetilde{\b{P}}^{2^{d-1}}}\ll \widetilde{\b{P}}^{-d}\sum_{-2\bfP < \bfh^{(1)} < 2\bfP}\ldots \sum_{-2\bfP< \bfh^{(d-1)}< 2\bfP} |S_w(\bfh^{(1)},\ldots, \bfh^{(d-1)},\bfalp)|,$$
where 
\begin{equation*}
\begin{split}
S_w&(\bfh^{(1)},\ldots, \bfh^{(d-1)},\bfalp)=\\ &\sum_{\bfy\in \Z^n}\left\{ \prod_{i=1}^n w_{h_i^{(1)}/P_i,\ldots ,h_{i}^{(d-1)}/P_i}\left(\frac{y_i}{P_i}-z_i\right) \right\}e\left(\sum_{i=1}^R \alp_i \Phi_i(\bfh^{(1)},\ldots, \bfh^{(d-1)},\bfy) + c(\bfh^{(1)},\ldots, \bfh^{(d-1)})\right),
\end{split}
\end{equation*}
with integers $c(\bfh^{(1)},\ldots, \bfh^{(d-1)})$ independent of $\bfy$. Hence
$$|S_w(\bfh^{(1)},\ldots, \bfh^{(d-1)},\bfalp)|=|\sum_{\bfy\in \Z^n}\left\{ \prod_{i=1}^n w_{h_i^{(1)}/P_i,\ldots ,h_{i}^{(d-1)}/P_i}\left(\frac{y_i}{P_i}-z_i\right) \right\}e\left(\Phi (\bfh^{(1)},\ldots, \bfh^{(d-1)},\bfy) \right)|.$$
The estimate 
$$S_w(\bfh^{(1)},\ldots, \bfh^{(d-1)},\bfalp)\ll \prod_{i=1}^n \min\left\{P_i, \Vert \Phi (\bfh^{(1)},\ldots, \bfh^{(d-1)},\bfe^{(i)})\Vert^{-1} \right\}$$
can then be obtained via partial summation.
\end{proof}

We define the counting function
$$M(\bfalp,\bfP):=\sharp\{-2\bfP\leq \bfh^{(i)}\leq 2\bfP, 1\leq i\leq d-1: \Vert \Phi (\bfh^{(1)},\ldots, \bfh^{(d-1)},\bfe^{(j)})\Vert < P_j^{-1}\forall 1\leq j\leq n\}.$$

As Lemma 3.2 is deduced from Lemma 3.1 in \cite{Dav59} we obtain the following lemma. 

\begin{lemma}
One has
$$|S_w(\bfalp)|^{2^{d-1}}\ll \widetilde{\b{P}}^{2^{d-1}-d+1+\eps}M(\bfalp,\bfP).$$
\end{lemma}

Next we need a version of Lemma 12.6 in \cite{MR2152164} which is modified for lopsided boxes. 

\begin{lemma}\label{lem3.3}
Let $L_1,\ldots, L_n$ be symmetric linear forms given by $L_i=\gam_{i1} u_1+\ldots + \gam_{in} u_n$ for $1\leq i\leq n$, i.e. such that $\gam_{ij}=\gam_{ji}$ for $1\leq i,j\leq n$. Let $a_1,\ldots, a_n>1$ be real numbers. We denote by $N(Z)$ the number of integers solutions $u_1,\ldots, u_{2n}$ of the system of inequalities
$$|u_i|< a_iZ,\quad 1\leq i\leq n,\quad |L_i- u_{n+i}|<a_i^{-1}Z,\quad 1\leq i\leq n.$$
Then for $0< Z_1\leq Z_2\leq 1$ we have
$$\frac{N(Z_2)}{N(Z_1)}\ll\left(\frac{Z_2}{Z_1}\right)^{n}.$$
\end{lemma}

\begin{proof}
Let $\Lam$ be the $2n$-dimensional lattice defined by
\begin{equation*}
\begin{split}
x_i&=a_i^{-1}u_i,\quad 1\leq i\leq n\\
x_{n+i}&= a_i(\gam_{i1} u_1+\ldots +\gam_{in}u_n+u_{n+i}),\quad 1\leq i\leq n.
\end{split}
\end{equation*}
As in the proof of Lemma 12.6 in \cite{MR2152164} we note that the inequalities describing $N(Z)$ are equivalent to
$$|x_i| < Z, \quad 1\leq i\leq 2n,$$
for a point $(x_1,\ldots, x_{2n})$ in the lattice $\Lam$. We identify the lattice $\Lam$ with its matrix
\begin{equation*}
\Lam=\left(\begin{array}{cccccc}a_1^{-1}&\ldots & 0 & 0 & \ldots & 0\\ \vdots & &\vdots&\vdots&&\vdots \\ 0 & \ldots & a_n^{-1} & 0& \ldots &0 \\a_1\gam_{11}&\ldots & a_1\gam_{1n} & a_1& \ldots & 0\\ \vdots & &\vdots&\vdots&&\vdots \\ a_n \gam_{n1} & \ldots & a_n \gam_{nn} & 0 & \ldots & a_n\end{array}\right)
\end{equation*}
and we find that the adjoint lattice is given by
\begin{equation*}
M=(\Lam^t)^{-1}=\left(\begin{array}{cccccc}a_1&\ldots & 0 & -a_1\gam_{11} & \ldots & -a_1\gam_{n1}\\ \vdots & &\vdots&\vdots&&\vdots \\ 0 & \ldots & a_n & -a_n\gam_{1n}& \ldots &-a_n\gam_{nn} \\ 0&\ldots & 0 & a_1^{-1}& \ldots & 0\\ \vdots & &\vdots&\vdots&&\vdots \\ 0 & \ldots & 0 & 0 & \ldots & a_n^{-1}\end{array}\right).
\end{equation*}
Since $\gam_{ij}=\gam_{ji}$ for all $1\leq i,j\leq n$ the two lattices $\Lam$ and $M$ can be transformed into one another by interchanging the order of $x_1,\ldots, x_{2n}$ and $u_1,\ldots, u_{2n}$ and changing signs at some variables. Hence they have the same successive minima. Now the proof of Lemma 12.6 in \cite{MR2152164} applies to our situation and 
an identical argument concludes our proof.
\end{proof}

We now apply Lemma \ref{lem3.3} to the counting function $M(\bfalp,\bfP)$. Let $0<\tet<1$ and set $Z=P_{\max}^{\tet-1}$. We then obtain the following bound,
$$|S_w(\bfalp)|^{2^{d-1}}\ll
 \frac{
\widetilde{\b{P}}
^{2^{d-1}-d+1+\eps}}{Z^{(d-1)n}}\sharp\calI,$$
where $\calI$ is defined by
$$
\Big\{
(\bfx^{(1)},\ldots, \bfx^{(d-1)})\in \Z^{(d-1)n}: |\bfx^{(i)}|\leq Z\bfP, \Vert \Phi(\bfx^{(1)},\ldots, \bfx^{(d-1)},\bfe_j)\Vert < Z^{d-1}P_j^{-1},\, \forall 1\leq j\leq n
\Big\}.$$
We are now in a position to obtain a form of Weyl's inequality for $S_w(\bfalp)$ (for a Weyl's inequality in a similar setting see for example Lemma 4.3 in \cite{Bir62}). Let $V^*$ be the affine variety defined by
$$\rank \left(\frac{\partial g_i^\natural(\bfx)}{\partial x_j}\right)_{1\leq i\leq R, 1\leq j\leq n}<R,$$
and recall that $$K=\frac{n-\dim V^*}{2^{d-1}}.$$

\begin{lemma}\label{lemWeyl}
Assume that $0<\tet <1$. Then one has either\\
(i) $$S_w(\bfalp) \ll \widetilde{\b{P}}^{1+\eps} \left(\frac{P_{\max}}{P_{\min}}\right)^{K}P_{\max}^{-\tet K},$$ or\\
(ii) there are integers $1\leq q\leq \Vert \bfg^\natural\Vert^RP_{\max}^{R(d-1)\tet}$, and $0\leq a_1,\ldots, a_R < q$ with $\gcd(\bfa,q)=1$ and
$$|q\alp_i-a_i|\leq \Vert \bfg^\natural\Vert^{R-1} P_{\min}^{-1}P_{\max}^{-(d-1)+R(d-1)\tet},\quad 1\leq i\leq R.$$
\end{lemma}

\begin{proof}
First assume that $P_{\max}^{\tet -1}P_{\min} \geq 1$.
We start with the bound
$$|S_w(\bfalp)|^{2^{d-1}}\ll 
\widetilde{\b{P}}^{2^{d-1}-d+1+\eps}P_{\max}^{(1-\tet)(d-1)n}\sharp\calI.$$
Consider the affine variety $\calY \subset \A^{n(d-1)}$ given by
$$\calY: \rank (\Phi_i(\bfx^{(1)},\ldots, \bfx^{(d-1)},\bfe_j))_{1\leq i\leq R, 1\leq j\leq n}<R.$$
We set
$$\calE:=\{(\bfx^{(1)},\ldots, \bfx^{(d-1)})\in \Z^{n(d-1)}\cap \calY: |\bfx^{(i)}|\leq P_{\max}^{\tet-1}\bfP,\, \forall 1\leq i\leq d-1\}.$$
Now we distinguish two cases.\vspace{0.5cm}\\
(i) Assume that $\calI\subset \calE$. Then we bound the cardinality of $\calE$ by dimension bounds. We dissect the region given by the conditions that $|\bfx^{(i)}|\leq P_{\max}^{\tet-1}\bfP$ into boxes where all the side length are equal (at the boundaries we allow for overlapping boxes which will result in slight overcounting) and of size $P_{\max}^{\tet -1}P_{\min}$. The number of such boxes is bounded by
$$\ll \left(\prod_{i=1}^n \frac{P_i}{P_{\min}}\right)^{d-1}.$$
On each of the boxes we apply a linear transformation to move the box to the origin. Then we apply Theorem 3.1 in \cite{MR2559866}. Note that this bound is independent of the coefficients of the variety (only depending on the dimension and degree) and hence uniform in the shift. We obtain
$$\sharp\calE \ll \left(\prod_{i=1}^n \frac{P_i}{P_{\min}}\right)^{d-1} (P_{\max}^{\tet -1}P_{\min})^{\dim \calY}.$$
By~\cite[Lem.3.3]{Bir62} we have
$\dim \calY \leq \dim V^* + (d-2)n$,
hence we obtain the bound
$$\sharp\calE \ll \left(\prod_{i=1}^n \frac{P_i}{P_{\min}}\right)^{d-1} (P_{\max}^{\tet -1}P_{\min})^{\dim V^*+(d-2)n}.$$
Together with our assumption $\calI\subset \calE$ we obtain
\begin{equation*}
\begin{split}
|S_w(\bfalp)|^{2^{d-1}}& \ll
\widetilde{\b{P}}^{2^{d-1}+\eps}P_{\max}^{(\tet-1)(\dim V^*+(d-2)n)}P_{\max}^{-(d-1)n(\tet-1)}P_{\min}^{-n+\dim V^*} \\
& \ll \widetilde{\b{P}}^{2^{d-1}+\eps} P_{\max}^{(1-\tet)(n-\dim V^*)}P_{\min}^{-n+\dim V^*} \\
&\ll \widetilde{\b{P}}^{2^{d-1}+\eps}  P_{\max}^{-\tet (n-\dim V^*)}\left(\frac{P_{\max}}{P_{\min}}\right)^{n-\dim V^*}.
\end{split}
\end{equation*}
This estimate gives option (i) in the statement of our lemma.\vspace{0.3cm}\\
Next we assume that $\calI \setminus  \calE\neq \emptyset$. Let $(\bfx^{(1)},\ldots, \bfx^{(d-1)})$ be such a point in the difference set, i.e. 
$$\rank (\Phi_i(\bfx^{(1)},\ldots, \bfx^{(d-1)},\bfe_j))_{1\leq i\leq R, 1\leq j\leq n}=R.$$
With no loss of generality
we assume that the leading $R\times R$ minor is of full rank, and set
$$q:=
\big|
\det(\Phi_i(\bfx^{(1)},\ldots, \bfx^{(d-1)},\bfe_j))_{1\leq i,j\leq R}
\big|
.$$
Note that
$$q\ll \Vert \bfg^\natural\Vert^R P_{\max}^{R(d-1)\tet}.$$
Moreover, we have the system of equations
$$\sum_{i=1}^R \alp_i \Phi_i(\bfx^{(1)},\ldots, \bfx^{(d-1)},\bfe_j) = \widetilde{a}_j+\del_j, \quad 1\leq j\leq R,$$
with $\widetilde{a}_1,\ldots, \widetilde{a}_R$ integers and 
$$|\del_j|\ll P_{\max}^{(\tet-1)(d-1)}P_j^{-1},\quad 1\leq j\leq n.$$
We now obtain (after changing $\tet$ by $\eps$ for $\eps$ arbitrarily small) as in the proof of~\cite[Lem.2.5]{Bir62} an approximation $1\leq a_1,\ldots, a_R\leq q$ to the real numbers $\alp_i$ of the quality
$$|q\alp_i-a_i|\leq  \Vert \bfg^\natural\Vert^{R-1} P_{\min}^{-1}P_{\max}^{-(d-1)+R(d-1)\tet},\quad 1\leq i\leq R.$$
Note that alternative (i) in Lemma \ref{lemWeyl} trivially holds if $P_{\max}^{\tet-1}P_{\min} \leq 1$.
\end{proof}

Next we come to the definition of the major arcs. Let $0<\tet<1$ and assume that
\begin{equation}\label{eqntet}
P_{\max}^{\tet-1}P_{\min}\geq 1.
\end{equation}
For $q\in \N$ and $1\leq a_1,\ldots, a_R\leq q$ we define the major arc
$$\grM_{\bfa,q}(\tet):=\{\bfalp\in [0,1]^R: |q\alp_i-a_i|\leq \Vert \bfg^\natural\Vert^{R-1} P_{\min}^{-1}P_{\max}^{-(d-1)+R(d-1)\tet},\, 1\leq i\leq R\}.$$
Moreover we define the major arcs $\grM(\tet)$ as the union
$$\grM(\tet)=\bigcup_{1\leq q\leq \Vert\bfg^{\natural}\Vert^R P_{\max}^{R(d-1)\tet}}\bigcup_{\substack{1\leq a_1,\ldots, a_R\leq q\\ \gcd(\bfa,q)=1}}\grM_{\bfa,q}(\tet)$$
and
set $\grm(\tet):=[0,1]^R\setminus \grM(\tet)$.

A short calculation gives the following bound for the measure of the major arcs $\grM(\tet)$.

\begin{lemma}\label{majarc}
Assume that $0<\tet<1$ such that (\ref{eqntet}) holds. Then one has
$$\meas(\grM(\tet))\ll \Vert \bfg^\natural\Vert^{R^2} P_{\min}^{-R}P_{\max}^{-R(d-1)+R(R+1)(d-1)\tet}.$$
\end{lemma}

We are now ready to provide an $L^1$-bound for the exponential sum $S_{w}(\bfalp)$ over the minor arcs, which is a modification of Lemma 4.4 in \cite{Bir62} and proved in the very same way.

\begin{lemma}\label{lem3.7}
Let $0<\tet<1$ such that (\ref{eqntet}) holds. Assume that
$$K>R(R+1)(d-1).$$
Then one has
$$\int_{\grm(\tet)}|S_{w}(\bfalp)|\d\bfalp\ll 
\widetilde{\b{P}}^{1+\eps} \left(\frac{P_{\max}}{P_{\min}}\right)^{K}P_{\max}^{- K}+\widetilde{P}^{1+\eps} \Vert\bfg^\natural\Vert^{R^2} \left(\frac{P_{\max}}{P_{\min}}\right)^{R+K}P_{\max}^{-Rd -(K-R(R+1)(d-1))\tet +\eps},$$
for $\eps> 0$ arbitrarily small. 
\end{lemma}

For technical convenience we introduce the slightly larger major arcs
$$\grM'_{\bfa,q}(\tet):=\{\bfalp\in [0,1]^R: |q\alp_i-a_i|\leq q\Vert \bfg^\natural\Vert^{R-1} P_{\min}^{-1}P_{\max}^{-(d-1)+R(d-1)\tet},\, 1\leq i\leq R\},$$
and
$$\grM'(\tet)=\bigcup_{1\leq q\leq \Vert\bfg^{\natural}\Vert^R P_{\max}^{R(d-1)\tet}}\bigcup_{\substack{1\leq a_1,\ldots, a_R\leq q\\ \gcd(\bfa,q)=1}}\grM'_{\bfa,q}(\tet).$$

We record that the major arcs $\grM'_{\bfa,q}(\tet)$ are disjoint for $\tet$ small enough and that
$$\meas(\grM'(\tet))\ll \Vert\bfg^\natural\Vert^{2R^2}P_{\min}^{-R}P_{\max}^{-R(d-1)+(2R^2+R)(d-1)\tet}.$$

A minor modification of the proof of Lemma 4.1 in \cite{Bir62} gives the following result.

\begin{lemma}\label{lemmajdis}
Assume that
\begin{equation}\label{majdis}
\Vert \bfg^\natural\Vert^{3R-1} P_{\min}^{-1}P_{\max}^{-(d-1)+3R(d-1)\tet} <1.
\end{equation}
Then for $1\leq q\leq \Vert\bfg^{\natural}\Vert^R P_{\max}^{R(d-1)\tet}$ and $1\leq a_1,\ldots, a_R\leq q$, $\gcd(\bfa,q)=1$ the major arcs $\grM'_{\bfa,q}(\tet)$ are disjoint.
\end{lemma}

We now come to the major arc approximation of $S_w(\bfalp)$. Let $q\in \N$ and $1\leq a_1,\ldots, a_R\leq q$. We define the exponential sum
$$S_{\bfa,q}:=\sum_{\bfy \mod q}e\left(\frac{\bfa}{q}\cdot\bfg(\bfy)\right)$$
and the integral
$$I_w(\bfgam):=\int_{\R^n} e(\bfgam\cdot\bfg(\bfu)) \prod_{i=1}^n w \left(\frac{u_i}{P_i}-z_i\right)\d\bfu.$$

\begin{lemma}\label{lemmajapprox}
Let $q\in \N$ and $0\leq a_1,\ldots, a_R <q$. Write $\bfalp=\bfa/q+\bfbet$. Assume that $q<P_{\min}P_{\max}^{-\vareps}$ and  
$$|\bfbet|
q P_{\max}^{d-1}\Vert \bfg\Vert < P_{\max}^{-\vareps}.$$
Then one has the following approximation
for any real $N \geq 1$,
$$S_w(\bfalp)=q^{-n}S_{\bfa,q}I_w(\bfbet) + O_{N}(\widetilde{\b{P}}P_{\max}^{-N}).$$
\end{lemma}

\begin{proof}
We recall the definition of the exponential sum $S_w(\bfalp)$ as
$$S_w(\bfalp)=\sum_{\bfx\in \Z^n}\prod_{i=1}^n w\left(\frac{x_i}{P_i}-z_i\right)e(\bfalp\cdot\bfg(\bfx)).$$
We split the summation variables $\bfx$ into residue classes modulo $q$ and obtain
$$S_w(\bfalp)=\sum_{\bfy\mod q}e\left(\frac{\bfa}{q}\cdot\bfg(\bfy)\right)\sum_{\bfw\in \Z^n}\prod_{i=1}^n w\left(\frac{y_i+w_iq}{P_i}-z_i\right)e(\bfbet\cdot\bfg(\bfy+q\bfw)).$$
We now consider the inner sum for a fixed vector $\bfy$ modulo $q$. Let
$$\psi(\bfw):=\prod_{i=1}^n w\left(\frac{y_i+w_iq}{P_i}-z_i\right)e(\bfbet\cdot\bfg(\bfy+q\bfw)).$$
We apply Euler--Maclaurin's summation formula (see Theorem B.5 in \cite{MR2378655}) 
of order $\tilde{\kappa}$ into each coordinate direction. If we choose $\tilde{\kappa}$ large enough depending only on $\vareps$, $n$ and $N$ we obtain
$$\sum_{\bfw\in \Z^n}\psi(\bfw)= \int_{\bfw\in \R^n}\psi(\bfw)\d\bfw + O_{N}(\widetilde{\b{P}}
P_{\max}^{-N}).$$
Note that all the boundary terms in Euler--Maclaurin's summation formula vanish due to the smooth weight function $w$. Since $N$ was arbitrary we find after even enlarging $\tilde{\kappa}$ that
$$S_w(\bfalp)=S_{\bfa,q}\int_{\bfw\in \R^n}\psi(\bfw)\d\bfw + O_{N}(\widetilde{\b{P}}P_{\max}^{-N}).$$
A variable substitution now gives the statement of the lemma. 
\end{proof}

Next we consider the singular integral. Note that in contrast to most approaches we defined the integral $I_w(\bfgam)$ with the inhomogeneous polynomials $\bfg(\bfy)$ instead of taking their homogenizations. We now replace $\bfg(\bfy)$ by $\bfg^\natural(\bfy)$ in $I_w(\bfgam)$ which will simplify the discussion of absolute convergence. Define
$$I_w^\natural(\bfgam)=\int_{\R^n}
e(\bfgam\cdot\bfg^\natural(\bfu))
\prod_{i=1}^n w\left(\frac{u_i}{P_i}-z_i\right) 
\d\bfu.$$

\begin{lemma}
\label{lem:hcab1068}
Assume that $ |\bfz| \leq 1$. Then one has
$$I_w(\bfgam)-I_w^\natural(\bfgam) \ll \widetilde{\b{P}}
|\bfgam|
 \Vert \bfg\Vert 
P_{\max}^{d-1}.$$
\end{lemma}
The proof of the lemma follows from directly comparing the integrands of the two integrals. 
Under the assumptions of Lemma \ref{lemmajapprox} we observe that
\[
S_w(\bfalp)=q^{-n}S_{\bfa,q}I_w^\natural(\bfbet) +O_{N}(\widetilde{\b{P}}
P_{\max}^{-N})+O(\widetilde{\b{P}}
|\bfbet|
\Vert\bfg\Vert 
P_{\max}^{d-1}).
\]
We define the truncated singular series
\[
\grS(Q):=\sum_{q\leq Q} q^{-n} S_{\bfa,q},
\]
and the truncated singular integral
\[
J_w(Q):=\int_{|\bfgam|
\leq Q}I_w^\natural(\bfgam)\d\bfgam.
\]

With these definitions we can write the major arc contribution in the following way.

\begin{lemma}
Assume 
$|\bfz| \leq P_{\max}$
and
that (\ref{majdis}) holds, as well as
\begin{equation}\label{majapprox}
\max\{\Vert\bfg^\natural\Vert^R P_{\max}^{R(d-1)\tet}P_{\min}^{-1},\Vert\bfg\Vert\Vert\bfg^\natural
\Vert^{2R-1} P_{\min}^{-1} P_{\max}^{2R(d-1)\tet} \}< P_{\max}^{-\vareps}.
\end{equation}
Then one has
\begin{equation*}\begin{split}
\int_{\grM'(\tet)} S_w(\bfalp)\d\bfalp =& \grS\left(\Vert\bfg^\natural\Vert^R P_{\max}^{R(d-1)\tet}\right) J_w (\Vert\bfg^\natural\Vert^{R-1}P_{\min}^{-1}P_{\max}^{-(d-1)+R(d-1)\tet}) + O_{N}(\widetilde{\b{P}}
P_{\max}^{-N}) \\&+O\left(\Vert\bfg^\natural\Vert^{2R^2+R}\Vert\bfg\Vert
\widetilde{\b{P}} P_{\min}^{-R-1}P_{\max}^{-R(d-1)+(2R^2+2R)(d-1)\tet}\right),
\end{split}
\end{equation*}
for any real $N\geq 1$. 
\end{lemma}
\begin{proof}
By Lemma~\ref{lemmajdis} the major arcs are disjoint
thus the proof
follows from Lemma \ref{lemmajapprox}.
\end{proof}

Next we aim to complete the singular series. We recall Lemma 2.2 from \cite{JWA} (see also \cite[Section 2, Lemma 2.14]{JW}).

\begin{lemma}\label{expsum}
For any $\vareps>0$ one has
$$|S_{\bfa,q}|\ll \Vert\bfg^\natural\Vert^{K/(d-1)}q^{n-K/R(d-1)+\vareps}.$$
\end{lemma}
We shall soon see that 
the truncated singular series $\grS(Q)$ is converging for $Q\rightarrow \infty$,
thus we shall set
\begin{equation}
\label{def:convser}
\grS = \lim_{Q\rightarrow \infty} \grS(Q).
\end{equation}
Lemma \ref{expsum} gives the following speed of convergence.

\begin{lemma}\label{lem3.13}
Assume that $K>R(d-1)$. Then $\grS$ is absolutely convergent. Moreover one has
$$\grS-\grS(Q)\ll \Vert\bfg^\natural\Vert^{K/(d-1)} Q^{1-K/R(d-1)+\vareps},$$
for any $\vareps >0$ and $|\grS|\ll  \Vert\bfg^\natural\Vert^{K/(d-1)}$.
\end{lemma}

In preparation for the proof of the absolute convergence of the singular integral, we note the following lemma, which is a consequence of Lemma \ref{lemWeyl}.

\begin{lemma}\label{lem3.14}
Assume that $|\bfalp|^3
\Vert\bfg^\natural\Vert^2 P_{\min}^2P_{\max}^{2(d-1)}<1$. Then one has
$$S_w(\bfalp) \ll \widetilde{\b{P}}^{1+\eps} \left(\frac{P_{\max}}{P_{\min}}\right)^K \left(|\bfalp| 
\Vert\bfg^\natural\Vert^{-R+1}P_{\min}P_{\max}^{d-1}\right)^{-K/R(d-1)},$$
for any positive $\eps$.
\end{lemma}

\begin{lemma}\label{lem3.15} 
Assume that $P_i\geq 1$ for $1\leq i\leq n$ and that $ |\bfz| \leq 1$. 
Then
$$I_w^\natural(\bfgam)\ll
\widetilde{\b{P}}
\min\Bigg\{1,\widetilde{\b{P}}^\eps \left(\frac{P_{\max}}{P_{\min}}\right)^{K(1+1/R(d-1))} \left(P_{\max}^d|\bfgam|
\Vert\bfg^\natural\Vert^{-R+1}\right)^{-K/R(d-1)}\Bigg\}.$$
\end{lemma}

The proof of Lemma \ref{lem3.15} is relatively standard (see Lemma 5.2 in \cite{Bir62}), with the exception that we compare the oscillatory integral $I_w^\natural(\bfgam)$ with parameters $P_1,\ldots, P_n$ to an exponential sum with box length $B_1,\ldots, B_n$ such that $\frac{B_i}{B_{\max}}=\frac{P_i}{P_{\max}}$ for all $1\leq i\leq n$.

We shall show that the truncated singular integral $J_w(Q)$ converges for $Q\rightarrow \infty$,
we will therefore let
\begin{equation}
\label{def:convinte}
J_w:=\lim_{Q\rightarrow \infty} J_w(Q),
\end{equation}
and call it the singular integral.\\
In the following we will always assume that $1\leq P_i$ for $1\leq i\leq n$ and that $ |\bfz|
\leq 1$. As a consequence of Lemma \ref{lem3.15} we obtain the following result. 

\begin{lemma}\label{lem3.16}
Assume that $K>R^2(d-1)$. Then $J_w$ is absolutely convergent and  
$$J_w -J_w(Q)\ll \widetilde{\b{P}}^{1+\eps}\left(\frac{P_{\max}}{P_{\min}}\right)^{K(1+1/R(d-1))} \left(P_{\max}^d \Vert\bfg^\natural\Vert^{-R+1}\right)^{-K/R(d-1)}Q^{-K/R(d-1)+R}.$$
Moreover, we have
$$J_w \ll \widetilde{\b{P}}^{1+\eps} \left(\frac{P_{\max}}{P_{\min}}\right)^{R^2(d-1)+R}P_{\max}^{-Rd}\Vert\bfg^\natural\Vert^{R(R-1)}.$$
\end{lemma}

We can now complete both the singular series and singular integral in our major arc analysis. According to Lemma \ref{lem3.13} and Lemma \ref{lem3.16} we obtain the following result.

\begin{lemma}\label{lem3.17}
Assume that equations 
(\ref{eqntet}), (\ref{majdis}) and (\ref{majapprox})
 hold and $ |\bfz| \leq P_{\max}$, as well as $K>R^2(d-1)$.
Then the following holds
for any real $N\geq 1$,
\begin{equation*}\begin{split}
\int_{\grM'(\tet)} S_w(\bfalp)\d\bfalp =& \grS J_w 
+O\left(\Vert\bfg^\natural\Vert^{2R^2+R}\Vert\bfg\Vert
\widetilde{\b{P}}
P_{\min}^{-R-1}P_{\max}^{-R(d-1)+(2R^2+2R)(d-1)\tet}\right) \\ &
+O_{N}\left(
\widetilde{\b{P}}
P_{\max}^{-N}+ \Vert\bfg^\natural\Vert^{K/(d-1)+R^2-R} 
\widetilde{\b{P}}^{1+\eps}\left(\frac{P_{\max}}{P_{\min}}\right)^{R+K} P_{\max}^{-Rd-K\tet+R^2(d-1)\tet} \right).
\end{split}
\end{equation*}
\end{lemma}

Theorem \ref{thm:master2} is now a consequence of the major arc analysis in Lemma \ref{lem3.17} in combination with the minor arc analysis from Lemma \ref{lem3.7}. For this, we choose $\tet$ by
$$P_{\max}^\tet = \Vert \bfg\Vert^{-\frac{1}{2R(d-1)+1}}\Vert \bfg^\sharp\Vert^{-\frac{3R}{3R(d-1)+1}} P_{\max}^{\frac{1}{4R(R+1)d}}.$$
Then we clearly have $0<\tet<1$ and equation (\ref{eqntet}) reduces to the assumption (\ref{cond:master}). Moreover, one quickly sees that with this choice of $\tet$ both of the conditions (\ref{majdis}) and (\ref{majapprox}) are satisfied. It remains to understand that the error terms in Lemma \ref{lem3.17} and Lemma \ref{lem3.7} are both majorised by the error term in Theorem \ref{thm:master2}. We bound the first error term in Lemma \ref{lem3.17} by
\begin{equation*}
\begin{split}
\Vert\bfg^\sharp\Vert^{2R^2+R}\Vert\bfg\Vert \widetilde{\b{P}} \left(\frac{P_{\max}}{P_{\min}}\right)^R P_{\max}^{-Rd-1}P_{\max}^{(2R^2+2R)d\tet}  \ll  \widetilde{\b{P}} \left(\frac{P_{\max}}{P_{\min}}\right)^R P_{\max}^{-Rd-1/2}.
\end{split}
\end{equation*}
Note that the last error term in Lemma \ref{lem3.17} as well as the second error term in Lemma \ref{lem3.7} are bounded by 
$$\Vert\bfg^\natural\Vert^ 
{\frac{2K}{d-1}-R}
\Vert\bfg\Vert^
{\frac{K-R^2(d-1)}{2R(d-1)}}  
\widetilde{\b{P}}^{1+\eps}
\left(P_{\max}/P_{\min}\right)^{R+K} 
P_{\max}^{-Rd-\frac
{K-R(R+1)(d-1)}
{4R(R+1)d} 
}.
$$
The first error term in Lemma \ref{lem3.7} is also present in the statement of Theorem \ref{thm:master2}.

Lastly let us remark that it is a well-known fact that the singular series factorises as
$\grS=\prod_p\sigma_p(\bfg)$,
where for any prime $p$ we
have
\[\sigma_p(\bfg):=
\lim_{l\rightarrow \infty} p^{-l(n-R)} \sharp
\big\{
1\leq \bfx \leq p^l : p^l| \bfg(\b{x})
\big\}
.\]
\section{Local densities}
\label{s:localden}
Throughout this section we will have $R$ forms 
of degree $d> 1$,
\[
f_1,\ldots,f_R
\in \Z[x_1,\ldots,x_n]
\]
and we will always assume that the Birch rank satisfies
\[
\mathfrak{B}(\bff)>2^{d-1}(d-1)R(R+1)
.\]
For a prime $p$
and a vector $\bfj=(j_1,\ldots,j_n) \in (\Z_{\geq 0})^n$
we shall be concerned with bounding the quantities
\[\delta(\bfj):=
\lim_{l\rightarrow \infty} p^{-l(n-R)} \sharp
\big\{
1\leq x_1,\ldots,x_n \leq p^l : p^l| \bff(p^{j_1}x_1,\ldots,p^{j_n}x_n)
\big\}
,\]
these estimates will be applied later towards the proof of Theorems~\ref{thm:mainvector},~\ref{thm:rosiwa} 
and~\ref{thm:weighted}.
We suppress the letter
$p$ from the notation for $\delta$ to make the presentation easier to follow.
The forms $\bff$ will be considered constant, however the prime $p$ and the vector $\bfj$ will not,
thus we shall require uniformity of our bounds with respect to $p$ and $\bfj$.
For later applications we only have to consider all big enough primes $p>z_0$,
where $z_0$ is a constant depending at most on the coefficients of $\bff$ and $n,d,R$.
This constant will be enlarged, if needed,
with no further comment.
Let us emphasize that the entities $\delta(\bfj)$ encode
the probability of the events
$$p^{j_1}|x_1,\ldots,p^{j_n}|x_n$$
as $\b{x}\in \Z^n$ sweeps through the zeros of $\bff=\b{0}$,
therefore, they are intimately connected with certain closed subvarieties of
$\b{f}=\b{0}$. This is manifested even in the most simple of situations: for
a primitive integer zero of $x_1x_2=x_3^2$ and a prime $p|x_3$ we always have $p^2|x_1$ or $p^2|x_2$ as a result of the subvariety $x_1x_2=x_3^2, x_3=0$
being reducible. We shall give geometric conditions that prevent $\delta(\bfj)$ 
to attain large values for general systems $\bff=\b{0}$.

For every $\bfj \in \{0,1\}^n$ we define the system $\b{f}^{\bfj}=\b{0}$
of $R$ forms  in $n-|\b{j}|_1$ variables via
$$f_\xi^{\bfj}(\bfx)=f_\xi(x_1,\ldots, x_n)|_{x_i=0 \mbox{ if } j_i=1},
\
 \xi \in \N\cap [1,R].$$
We later need a lower bound for the Birch rank of the new systems, as for example obtained in~\cite[Lem.3]{CM}.
As there is a slight oversight in the proof of \cite[Lem.3]{CM}, we give here the statement and proof of the corrected lemma where the quantity $R$ in \cite[Lem.3]{CM} is replaced by $R+1$.

\begin{lemma}\label{e:ccmm}
One has
\begin{equation}
\label{e:ccmm}
\mathfrak{B}(\bff^{\bfj})\geq \mathfrak{B}(\bff) - (R+1)|\b{j}|_1.
\end{equation}
It is important to note here that we view $\bff^{\bfj}$ as a system of $R$ equations in $n-|\bfj|_1$ variables.
\end{lemma}
\begin{proof}
Let $\widetilde{V^*}\subset \P_\C^{n-1}$ be the projective variety given by 
$$\rank \left(\frac{\partial f_\xi}{\partial x_i}\right)_{\xi,i}<R,$$
and note that this is well-defined as all of the polynomials $f_\xi$ are homogeneous. Then the Birch rank of $\bff$ is given by
$$\mathfrak{B}(\bff)=n-\dim (\widetilde{V^*})-1.$$
Similarly, let $\widetilde{V^{*,\bfj}}\subset \P_\C^{n-|\bfj|_1-1}$ be the projective variety given by 
$$\rank \left(\frac{\partial f_\xi^{\bfj}}{\partial x_i}\right)_{\xi,i}<R,$$
such that we have
$$\mathfrak{B}(\bff^{\bfj})=n-|\bfj|_1-\dim (\widetilde{V^{*,\bfj}})-1.$$
The variety $\widetilde{V^{*,\bfj}}$ naturally embeds into the linear subspace of $\P_\C^{n-1}$ given by $x_i=0$ for $j_i=1$. We write $\iota(\widetilde{V^{*,\bfj}})$ for this embedding. 
Then we observe that 
$$\iota(\widetilde{V^{*,\bfj}}) \cap \bigcap_{1\leq \xi\leq R} \bigcap_{\substack{1\leq i\leq n\\j_i=1}}\left\{\frac{\partial f_\xi}{\partial x_i}=0\right\}\subset \widetilde{V^*}.$$
Hence we obtain
$$\dim (\widetilde{V^{*,\bfj}}) - R|\bfj|_1\leq \dim (\widetilde{V^*} ) .$$
Finally, this implies
$$\mathfrak{B}(\bff^{\bfj})\geq n-|\bfj|_1-(\dim (\widetilde{V^*} ) + R|\bfj|_1)-1 = \mathfrak{B}(\bff) - (R+1)|\b{j}|_1.$$
\end{proof}

This is a convenient place to introduce the 
helpful notation
\[
\Theta(\bfj)
:=
\frac{\mathfrak{B}(\b{f}^\bfj)}{R (d-1)2^{d-1}}
\]
and $\Theta(\b{0})$
will be denoted by
$\Theta$. For non-negative integers $j_1,\ldots, j_n$, any prime $p$ 
and a vector $\b{x}$ we use the notation
\[
p^{\b{j}}|\b{x}
\iff
\ p^{j_i}|x_i, \forall 1\leq i\leq n
.\]
This enables us to introduce the densities 
$$\sig_p(p^{\b{j}}|\b{x})= \lim_{l\rightarrow \infty} p^{-l(n-R)} 
\sharp
\Big\{1\leq x_1,\ldots,x_n \leq p^l 
: p^l|
\bff(\bfx), 
p^{\bf{j}}|\bf{x}
\Big\}
$$
and from the definition of $\delta$
we infer that  
\[
\frac{\delta(\b{j})}
{
 p^{|\b{j}|_1}
}
= 
\sig_p(p^{\b{j}}|\b{x}).
\]
\begin{lemma}
\label{lem:triviallem}
Let $t,d$
be integers with $2\leq d <t$.
Then 
for each $\bfa \in (\Z/p^{t-d}\Z)^R$ with $p\nmid \bfa$
and any 
vector polynomial
$\b{g}
\in \Z[\b{x}]^R$
with $\max_{1\leq i\leq R}\deg(g_i) \leq d-1$
we have
\[
\sum_{\bfx \md{p^{t-1}}}
e_{p^{t}}\l(p^d
\b{a} \cdot \b{f}(\bfx)
+p\b{a} \cdot \b{g}(\bfx) 
\r)
\ll_\eps
p^{(t-1)(n-\frac{t-d}{t-1} \Theta
+\epsilon)}
,\]
where
the implied constant is independent of $p,t, \b{g}$ and $\bf{a}$.
\end{lemma} 
\begin{proof} 
We shall use~\cite[Lem.4.3]{Bir62}
with $P=p^{t-1}$ and $\boldsymbol{\alpha}=p^{-t+d} \b{a}$; 
in doing so we observe that lower degree polynomials leave
the strength of the bounds in~\cite{Bir62} unaffected.
Recall that the constant $K$ in~\cite[Eq.(8)]{Bir62}
is given via $\mathfrak{B}(\bff)/2^{d-1}$.
Our aim is to acquire
a constant $\eta>0 
$, as large as possible,
such that 
$\boldsymbol{\alpha} \notin \c{M}(\eta)$,
where
$\c{M}(\eta)$ is
given in~\cite[Sect.4,Eq.(5)]{Bir62}.
This would then imply 
that the sum in our lemma is 
\[
\ll 
p^
{(t-1)
(n-
\frac{\mathfrak{B}(\bff)\eta}{2^{d-1}}
+\eps
)}
.
\]
The assumption
$\boldsymbol{\alpha} \in \c{M}(\eta)$
provides
non-negative
 integers
 $q',a_1',\ldots,a_R'$
fulfilling 
$$\gcd(a_1',\ldots,a_R',q')=1,
\
1\leq q' \leq p^{(t-1)R(d-1)\eta}$$
and such that for all $i=1,\ldots,R$ 
the succeeding inequality is valid,
\begin{equation}
\label{eq:ineq0}
2|q'a_i
-a_i' p^{t-d}
|\leq
p^
{t-d
+
(t-1)(-d+R(d-1)\eta)}
.\end{equation}
As explained in~\cite[Lem.4.1]{Bir62},
we need to assume $2R(d-1)\eta <d$ in order 
to ensure that 
the major arcs are
disjoint.  
It is straightforward
to infer
that this condition is met 
upon choosing 
$$\eta := \eta(\eps)= \frac{t-d}{(t-1)R(d-1)}-\eps$$
for any
small enough $\eps>0$.
Furthermore, this choice of $\eta$ makes the exponent of $p$ in~\eqref{eq:ineq0}
non-positive,
thus giving birth to the equalities
$q'a_i=a_i' p^{t-d}$ for all $i$.
In particular, we obtain
$p^{t-d}=q'\leq p^{(t-1)R(d-1)\eta}$, thus
$
t-d 
\leq 
(t-1)R(d-1)\eta
$,
which constitutes a violation to the
the definition of $\eta$.
\end{proof}
For $\bfj \in \{0,1\}^n$, $c \in \N$ and any prime $p$
define 
$$E(p^c;\bfj):=\sharp\Big\{
\bfx \md{p^c}: \bff(\bfx)\equiv \mathbf{0}\md{p^c},
p^{j_i}| x_i \ \forall i
\Big\}.$$
This quantity is intimately related to the geometry of $\bff^{\bfj}=\b{0}$
and we begin by using it to approximate $\delta(\b{j})$. 
\begin{lemma}\label{lem4}
Let $\bfj
\in \{0,1\}^n$ 
and assume that
$\Theta>R$.
Then there is some $z_0>0$, such that for $p>z_0$ and each sufficiently small
$\eps>0$,
we have  
$$
\delta(\bfj)   
=p^{d(R-n)+|\bfj|_1}E (p^d;\bfj)
+O(p^{
-\Theta
+R(d+1)
+\eps
})
,$$
where the implied constant depends at most on
$\b{f}$.
\end{lemma}
\begin{proof}
For $t\geq 1$, $\b{j} \in \{0,1\}^n$
and any $\b{a} \in \Z^R$ we 
bring into play the entities
$$
W_{\bfa,p^t}(p^{\b{j}}|\b{x})
:=\sum_{\substack{\bfx \mod{p^t}\\ p^{\b{j}}|\b{x}}}
e_{p^t}(\bfa\cdot \bff(\bfx))
\ \text{ and } \
G(\bfj;p^t)
:=
p^{-tn}
\Osum_{\bfa \mod{p^t}}
W_{\bfa,p^t}(p^{\b{j}}|\b{x})
,$$
where the summation $\Osum_{\!\!\!\b{a}\md{q}}$ is over vectors 
$\b{a} \in (\Z/q\Z)^R$ with $\gcd(\b{a},q)=1$.
We have 
\begin{equation*}
\begin{split}
\delta(\bfj) p^{-|\bfj|_1}
&= 
\lim_{l\rightarrow \infty}p^{-l(n-R)}
\sum_{\bfa \mod{p^l}}\frac{1}{p^{Rl}}
\sum_{\substack{\bfx \mod{p^l}\\ p^{\b{j}}|\b{x}}}
e_{p^l}(\bfa\cdot \bff(\bfx))\\ 
&=
\lim_{l\rightarrow \infty}
\sum_{\bfa \mod{p^l}}p^{-ln}
\sum_{\substack{\bfx \mod{p^l}\\ p^{\b{j}}|\b{x}}}
e_{p^l}(\bfa \cdot \bff(\bfx))\\
&= \lim_{l\rightarrow \infty}\left(\sum_{t=1}^l G(\bfj;p^t) + p^{-ln}\sharp\{\bfx\mod {p^l}: p^\bfj|\bfx\}\right)\\
&= p^{-|\bfj|_1}+ \lim_{l\rightarrow \infty}\sum_{t=1}^l G(\bfj;p^t),
\end{split}
\end{equation*}
whence
\begin{equation}
\label{eq:explicitBW}
\delta(\bfj)=1
+
p^{|\b{j}|_1}
\sum_{t=1}^\infty
G(\bfj;p^t).
\end{equation}
Observe that for each form $F \in \Z[\b{x}]$,
any prime $p$ and any fixed
integer vector $\bfy$
there exists an integer  polynomial $F_\bfy \in \Z[\bfx]$ 
of degree strictly smaller than $\deg(F)$,
such that  
\[
F(\bfy+p\bfx)=
p^{\deg(F)}F(\bfx)
+F(\bfy)
+pF_\bfy(\bfx)
.\]
Hence,
if $t\geq d+1$,
this
allows us to rewrite the exponential sum $W_{\bfa,p^t}(p^{\b{j}}|\b{x})$
as
\begin{align*} 
&\sum_{\substack{\bfy  \in (\N\cap [1,p])^n\\p^{\b{j}}|\b{y}}} 
\sum_{\bfh  \in (\N\cap [1,p^{t-1}])^n}  
e_{p^t}(\bfa\cdot \bff(\bfy+p\bfh))\\
=
&\sum_{\substack{\bfy  \in (\N\cap [1,p])^n\\p^{\b{j}}|\b{y}}}
e(p^{-t}\bfa\cdot \bff(\bfy))
\sum_{\bfh  \in (\N\cap [1,p^{t-1}])^n} 
e(p^{d-t}\bfa\cdot \bff(\bfh)+ p^{-t+1}\bfa\cdot \bfg_\bfy(\bfh)),
\end{align*}
where the polynomials $\bfg_\bfy(\bfh)$ have degree strictly smaller
than $d$ in $\bfh$. 
Invoking
Lemma~\ref{lem:triviallem}
endows
us
with
the following bound for the inner sum over $\b{h}$,
$$
\ll p^{(t-1)(n+\eps)-(t-d)\Theta},
$$
where
the implicit constant is
independent of $p, t, \bfy$ and $\bfa$. 
Hence, for $t>d$ we deduce that 
\begin{equation*}
\begin{split}
W_{\bfa,p^t}(p^{\bfj}|\bfx)
\ll p^{t(n+\eps)-|\b{j}|_1-(t-d)\Theta}
,\end{split}
\end{equation*}
thereby procuring the validity of 
\begin{equation*}
\begin{split} 
\sum_{t=d+1}^\infty |G(\bfj;p^t)|
\ll p^{-|\b{j}|_1+d\Theta}
\sum_{t=d+1}^\infty
p^{-t(\Theta-R-\eps)}
.\end{split}
\end{equation*}
Our assumption $R<\Theta$ shows that for each $0<\eps < (\Theta-R)/2$ 
the sum over $t$ has the value
\[
\frac{p^{-(d+1)(\Theta-R-\eps)}}
{1-p^{-(\Theta-R-\eps)}}
\]
and increasing the value of $z_0$ to ensure that 
$z_0^{(\Theta-R)/2} \geq 2$
shows that 
\[
\sum_{t=d+1}^\infty |G(\bfj;p^t)|
\ll p^{-|\b{j}|_1
-\Theta 
+R(d+1)
+\eps(d+1)
}
.\] 
To control the contribution of the terms with $t \leq d$ 
we
note that
$$p^{-|\b{j}|_1}+\sum_{t=1}^d G(\bfj;p^t) = p^{d(R-n)} 
\sharp\Big\{\bfx \mod{p^d}: \bff(\bfx)\equiv \mathbf{0} \mod{p^d},
p^{\b{j}}|\b{x}
\Big\},$$
thus concluding
our proof. 
\end{proof} 
Observe that, at least when $|\bfj|_1$ is relatively small, the quantity $E(p^d;\bfj)$ regards the
number of zeros $\md{p}$ of a variety in sufficiently many variables; thus the estimates of Birch yield
the required estimation of $E(p^d;\bfj)$.
\begin{lemma}\label{lem5}
Let $\bfj\in \{0,1\}^n$ and assume that 
$\Theta(\bfj) >  R$ is fulfilled.
Then  
for all $\eps >0$ and primes $p>z_0$ we have 
$$E(p^d;\bfj)=p^{d(n-R)-|\b{j}|_1}+O_\eps\left(p^{d(n-R)-|\b{j}|_1-(\Theta(\bfj)-R)+\eps}\right),$$
with an implicit constant that is independent of $p$.
\end{lemma}
\begin{proof} 
We initiate our argument
by
slicing the counting function $E(p^d;\bfj)$
along the variables which are divisible by $p$. 
Let $I=\{1\leq i\leq n: j_i =1\}$ 
and
for  
$\bfx'=(x_i)_{i\in I}\in (\Z/p^d\Z)^{|I|}$ we define
$$E(p^d;\bfj;\bfx'):=\sharp\Big\{
x_i \mod{p^d}, i\notin I: 
\bff(\bfx)\equiv \mathbf{0}\mod{p^d}\Big\}.$$
We rewrite this counting function with exponential sums as follows,
\begin{equation*}
\begin{split}
E(p^d;\bfj;\bfx')
= p^{d(n-|\b{j}|_1)-dR}+
p^{-dR}
\sum_{t=1}^d 
p^{(n-|\b{j}|_1)(d-t)}
\Osum_{\bfa \md{p^t}}  
\
\
\sum_{\substack{x_i\md{p^t}\\
i \notin I
}}
e_{p^t}(\bfa \cdot \bff(\bfx)).
\end{split}
\end{equation*}
Note that the degree $d$ part of the polynomial $\bff(\bfx)$ when viewed as a polynomial in the variables 
$x_i, 
i \notin I$,
is $\bff^{\bfj}(\bfx)$. 
We now apply~\cite[Lem.5.4]{Bir62},
the strength of which is unaffected by lower degree polynomials,
to obtain for any $\eps>0$ and uniformly for all $p>z_0$,
$$ \sum_{\substack{x_i\mod{p^t}\\
i \notin I
}}
e_{p^t}(\bfa\cdot \bff(\bfx))\ll_\eps 
p^{
t(n-|\b{j}|_1 - \Theta(\bfj))
+\eps
}.$$
We use this to estimate
$E(p^d;\bfj;\bfx')$ as follows,
\begin{equation*}
\begin{split}
E(p^d;\bfj;\bfx')-p^{d(n-|\b{j}|_1-R)}
&\ll_\eps p^{d(n-|\b{j}|_1-R)+\eps}
\sum_{t=1}^d p^{t(R-\Theta(\bfj))} 
\\&\ll_{\eps} p^{d(n-|\b{j}|_1-R)-(\Theta(\bfj)-R-\eps)}.
\end{split}
\end{equation*} 
We can now evaluate $E(p^d;\bfj)$ as
\[
\sum_{\substack{x_i\md{p^d}, i\in I \\p\mid x_i}}E(p^d;\bfj;\bfx')
=p^{dn-|\b{j}|_1-Rd}+O_\eps\l(p^{d(n-R)-|\b{j}|_1-(\Theta(\bfj)-R)+\eps}\r)
,\]
which concludes our proof. 
\end{proof}
Tying Lemmas~\ref{lem4} and~\ref{lem5} together
provides  the succeeding estimate.
\begin{corollary}
\label{cor1}
Assume that  
$\bfj \in \{0,1\}^n$,
$\min\{\Theta,\Theta(\bfj)\}>R$
and that $p$ is a prime in the range
$p>z_0$.
Then the following holds 
for each $\eps>0$ 
with an implied
constant depending only on $\bff$ and $\eps$,
$$
\delta(\bfj)
= 
1+O\!\l(
p^{R-\min\{\Theta-dR,\Theta(\bfj)\}+\eps}\r)
.$$
\end{corollary}  
Utilising~\eqref{e:ccmm} to find lower bounds for  $\Theta(\bfj)$
gives the following consequence of Corollary~\ref{cor1}.
\begin{corollary}
\label{cor2}
Assume that for some $\b{j}\in \{0,1\}^n$ 
we have 
\[
\mathfrak{B}(\b{f})
>
\max\big\{
(d-1)R^2
2^{d-1}
+(R+1)
|\b{j}|_1
,
(d^2-1)R^2
2^{d-1}
\big\}
.\] 
Then there exists $\lambda>0$ 
such that 
for all large enough primes $p>z_0=z_0(\bff)$,
we have 
$$
\delta(\bfj)
=1+O(p^{-\lambda})
,$$
with an implied
constant depending only on $\bff$.
\end{corollary}  
We can see that the bound  $\delta(\bfj)\ll 1$ fails when $|\bfj|_1$ approaches $n$ hence the assumption $\Theta(\bfj)>R$ of
Corollary~\ref{cor1} is no longer applicable.
Indeed, a moment's thought reveals that 
$\delta(1,\ldots,1)=p^{dR}\sigma_p$
and that whenever 
$h_i\geq j_i$
for all $1\leq i\leq n$
then
$\delta(\bfj)\geq \delta(\bfh) p^{|\bfj|_1-|\bfh|_1}$.
The bound $\sigma_p \gg 1$, valid with an implied constant independent of $p$
when $p$ is sufficiently large,
reveals that for such $p$ we have 
\[
n-\frac{dR}{2}<|\bfj|_1\leq n
\Rightarrow
\delta(\bfj)
\gg
p^{\frac{dR}{2}}
\]
with an implied constant independent of $p$.
Therefore we need to provide
(necessarily weaker)
bounds for the densities $\delta(\bfj)$
which are however valid through the whole range $1\leq |\bfj|_1\leq n$.
The crucial import will be bounds for the exponential sums in Birch's work 
with the additional property that the dependence on the coefficients of the 
underlying forms is explicitly recorded.
\begin{lemma}\label{lem6}
Assume that $\Theta>R$.
Then there exists a large $z_0=z_0(\bff)$
such that 
for 
each
$\bfj\in \{0,1\}^n$, $\eps>0$ and prime $p>z_0$ 
the following  holds
with an implicit constant depending at most on $\eps$ and $\bff$,
$$\delta(\bfj)\ll p^{dR \Theta+R-\Theta+\eps
}.$$
\end{lemma}
\begin{proof}
We start by rewriting 
\[
W_{\bfa,p^t}(p^{\b{j}}|\b{x})
=p^{-|\bfj|_1}
\sum_{\bfx\md{p^t}}e_{p^t}(\bfa\cdot\bff(p^{j_1}x_1,\ldots, p^{j_n}x_n))
\]
and
considering $\bff(p^{j_1}x_1,\ldots, p^{j_n}x_n)$ as a system of homogeneous polynomials in the variables $x_1,\ldots, x_n$. Note that the maximum of the coefficients is bounded by $C_1p^{d}$ for a positive constant 
$C_1=C_1(\bff)$ that is independent of $p$. Moreover, the Birch rank of 
the system $\bff(\bfx)=\b{0}$ equals the Birch rank of the system 
$\bff(p^{j_1}x_1,\ldots, p^{j_n}x_n)=\b{0}$. Alluding to the estimate~\cite[Lem.2.2]{JWA}
supplies us with the bound
$$ 
W_{\bfa,p^t}(p^{\b{j}}|\b{x})
p^{|\bfj|_1}
\ll_\eps p^{
dR\Theta
+t(n-\Theta
+\eps)},$$
which, once 
injected into~\eqref{eq:explicitBW},
offers the validity of
\begin{equation*}
\begin{split}
\delta(\bfj)-1
\ll 
p^{dR\Theta}
\sum_{t=1}^\infty p^{t(R-
\Theta
+\eps)}
.
\end{split}
\end{equation*}
Enlarging $z_0$ and $1/\epsilon$ if needed,
ensures the convergence of the sum over $t$
to a value that is $\ll_{z_0}p^{R-\Theta+\epsilon}$, independently of $p$.
\end{proof}

For a prime $p$ and a vector $\bfj \in (\Z_{\geq 0})^n$ we define 
\begin{equation}
\label{def:varpi1}
\varpi(p^{j_1},\ldots,p^{j_n}):=\frac{\delta(\bfj)}{\sigma_p(\bff)}
.\end{equation}
The standard estimate
$\sig_p=1+O(p^{-1-
\eps(\b{f})}
)$
holds for some $\eps(\b{f})>0$.
Alluding to~Lemma \ref{lem6}
supplies us with the following corollary. 
\begin{corollary}\label{corome}
Assume that
$
\mathfrak{B}(\b{f})
>
R^2 (d-1)2^{d-1}$ 
and recall the definition of $\Upsilon$ in~\eqref{def:theY}.
Then the following 
bound holds uniformly 
for each $\bfj \in \{0,1\}^n$ and $p> z_0$,
$$\varpi(p^{\bfj})\ll p^\Upsilon.$$
\end{corollary}

\section{Proof of Theorems~\ref{thm:rosiwa} and~\ref{thm:weighted}}
\label{s:firstsieve}
\subsection{Preparations}
\label{s:bwv1056largo}
Owing to~\eqref{cond:units},
there exists positive integers
$z_0=z_0(\b{f}),
m=m(\b{f})$  such that
if we let 
$$W:=\prod_{p\leq z_0}p^m,$$ then there exists  
$\b{y} \in (\N\cap [1,W])^n$ fulfilling the following,
\begin{equation}
\label{eq:wmtrick} 
\gcd(y_1\cdots y_n,W)=1
\end{equation}
and 
\begin{equation}
\label{cond:posit} 
p\leq z_0\Rightarrow\sigma_p(\bff(\bfy+W\b{x}))>0.
\end{equation} Define
\begin{equation}
\label{def:cala} 
\c{A}:=
 \{\b{x}\in \Z^n:
\b{f}(\b{x})=\b{0},\bfx\equiv \bfy\md{W}
 \}
.
\end{equation}

Let us now
choose a non-singular point
$\boldsymbol{\zeta}\in V_\b{f}(\R)$ (whose existence is guaranteed by~\eqref{cond:units})
and 
we
let $\eta \in (0,\min_i 
\{\min\{\zeta_i/2,(1-\zeta_i)/2\}
\})$ be arbitrary.
Defining  
\begin{equation} \label{def:bachbwv1043allegro}
\c{B}_\eta
:=\Big\{
\b{x}\in \R^n:
\Big|\b{x}-\frac{
\boldsymbol{\zeta}
}{
2|\boldsymbol{\zeta}|}
\Big| <\eta 
\Big\} 
,\end{equation}
we see that for any such $\eta$, one has 
$
\c{B}_\eta
\subset (0,1)^n$.
Now we choose any smooth function 
$w:\R\to\R_{\geq 0}$ 
of compact support in $[-\eta/2,\eta/2]$ 
and such that 
if 
$|t|\leq \eta/4$ then $w(t)>0$.  
Letting 
$w_0:=\sup\{w(t):t\in \R\}$
we have
$\mathbf{1}_{\{0<t\leq B\}}(t)
\geq  w_0^{-1}
w(
t/B
-\zeta_i/(2|\boldsymbol{\zeta}|)
)
$
and therefore for every $\b{x}\in \Z^n$, 
\begin{equation}
\label{eq:bwvfminor}
\prod_{i=1}^n
\mathbf{1}_{ \{0<x_i\leq B\}}(\b{x})
\geq  w_0^{-n}
\prod_{i=1}^n
w\bigg( \frac{x_i}{B} -\frac{\zeta_i}{2|\boldsymbol{\zeta}|} \bigg)
.\end{equation}

\subsection{A level of distribution result}
\label{s:lev123}
Let us now
take the opportunity to record a \textit{level of distribution} result that will be the main input in the forthcoming 
sieving arguments.
For $\bfk \in \N^n$ with $\gcd(\widetilde{\bfk},W)=1$
and each $k_i$ being square-free
let $w:\R\to \R_{\geq 0}$ be a smooth weight as above.
We let 
\begin{equation}
\label{def:nw}
N_w(B;\bfk)
:=\sum_{\substack{\b{x} \in \c{A}\\ k_i| x_i}}
\prod_{i=1}^n
w\Bigg(
\frac{x_i}{B}
-\frac{\zeta_i}{2|\boldsymbol{\zeta}|}
\Bigg)
.\end{equation}
Recall the definition of the matrix  $\boldsymbol{\epsilon}$ in~\eqref{def:epss}.
Our result will involve an error term 
related to the following function,
defined for $\b{m} \in \N^n$ and $B\geq 1$,
\[E(B;\b{m})
:=
\sum_{i=1}^3
B^{-\epsilon_{i,1}}
|\b{m}|^{\epsilon_{i,2}}
\min\{m_j\}^{\epsilon_{i,3}}
.\]
Furthermore, extend the function 
$\varpi$ defined in~\eqref{def:varpi1}
to $\N^n$ by
letting
for $\bfk \in \N^n$,
\[
\varpi(\b{k})
:=\prod_{p|
k_1\cdots k_n 
}
\varpi
\big(
p^{\nu_p(k_1)},\ldots,p^{\nu_p(k_n)}
\big)
\]
and if 
$\gcd(k_1\cdots k_n,W)=1$ we 
define 
$\boldsymbol{\tau} \in (\Z \cap [0,W))^n$ via 
$\langle \boldsymbol{\tau} \bfk \rangle \equiv \b{y}\md{W}$.
Finally, we 
let   
\[
\grS(\bff,W):=\prod_{p|W}\sigma_p(\b{f}(\boldsymbol{\tau}+W\b{s}))
\prod_{p\nmid W}\sigma_p(\bff)
\]
and
$$\mathcal{J}_w(\bff,W)
:= 
\frac{1}{W^n}
\int_{\R^R}
\int_{\R^n}
e\left( 
\boldsymbol{\gamma}
\cdot
\bff
\left(\b{u}\right) 
\right)
\prod_{i=1}^n 
w
\bigg(
u_i
- 
\frac{\zeta_i}{2|\boldsymbol{\zeta}|}
\bigg)
\d\b{u}
\d\boldsymbol{\gamma} 
.$$ 
\begin{lemma}
\label{lem:levofdi}
Assume $\mathfrak{B}(\bff)>2^{d-1}R(R+1)(d-1)$
and that 
$\bfk \in \N^n$ satisfies
$$
\gcd(k_1\cdots k_n,W)=1
\
\text{ and } \
|\bfk|\leq
B^{1/\rho}
(\log B)^{-1}
,$$
where
$B \in \mathbb{R}_{\geq 1}$ 
and
the constant $\rho$ was defined in~\eqref{def:rrr}.
Then for each $\eps>0$ we have
\[
N_w(B;\bfk)=
\mathcal{J}_w(\bff,W)
\grS(\bff,W)
\frac{\varpi(\bfk)}{\widetilde{\bfk}}
B^{n-Rd}
+O\bigg(
\frac{B^{n+\eps}}{\widetilde{\b{k}}}
E(B;\b{k})
\bigg)
.\]
\end{lemma}
\begin{proof}
Defining 
$\b{g}(\b{s}):=\b{f}(\langle \bfk(\boldsymbol{\tau} +W\b{s})\rangle )$
gives
\[
N_w(B;\bfk)=
\sum_{\substack{\b{s} \in \Z^n\\\b{g}(\b{s})=\b{0}}} 
\prod_{i=1}^n
w\Bigg(\frac{s_i}{\frac{B}{k_iW}}
-
\Bigg(
\frac{\zeta_i}{2|\boldsymbol{\zeta}|}-\frac{\tau_i}{\frac{B}{k_i}}
\Bigg)
\Bigg)
.\]
We shall apply Theorem~\ref{thm:master2} at this point;
before doing so we need to verify that 
\[
\Bigg|
\frac{\zeta_i}{2|\boldsymbol{\zeta}|}-\frac{\tau_i}{\frac{B}{k_i}}
\Bigg|
\leq 1
\]and that 
condition~\eqref{cond:master} is met.
The former is easy to verify due to $|\boldsymbol{\tau}|\leq W\ll 1$ and $\rho>1$, which implies that 
$B/k_i \geq B^{1-1/\rho} (\log B)\to +\infty$.
Regarding~\eqref{cond:master}, the obvious
equality
$\bfg^\natural(\b{s})=W^d\bff(\langle \bfk \b{s}\rangle)$
presents us with
$\max\big\{
\|\bfg^\natural\|,\|\bfg\|
\big\}
\ll
|\bfk|^d
$,
thus the growth condition on $|\bfk|$ in our lemma
is sufficient.
The last issue to be commented regards the 
real densities.
The real density provided by the application of Theorem~\ref{thm:master2}
is
\[\int_{\R^R}
\int_{\R^n}
e(
W^d
\boldsymbol{\beta}
\cdot
\bff(\langle \bfk \b{s}\rangle) 
)
\prod_{i=1}^n 
w\left(\frac{s_i}{\frac{B}{Wk_i}}
-\left(\frac{\zeta_i}{2|\boldsymbol{\zeta}|}
-\frac{\tau_i}{\frac{B}{k_i}}
\right)\right)
\d\b{s}\d\boldsymbol{\beta}
.\]
Note that the proof of Theorem \ref{thm:master2} in fact shows that the real density can also be replaced by its inhomogeneous version,
\[
\int_{\R^R}
\int_{\R^n}
e(
\boldsymbol{\beta}
\cdot
\bff(\langle \bfk 
\boldsymbol{\tau}
\rangle +W\langle \bfk \b{s}\rangle) 
)
\prod_{i=1}^n 
w\left(\frac{s_i}{\frac{B}{Wk_i}}
-\left(\frac{\zeta_i}{2|\boldsymbol{\zeta}|}
-\frac{\tau_i}{\frac{B}{k_i}}
\right)\right)
\d\b{s}\d\boldsymbol{\beta}
.\]
For this we note that the major arc analysis initially came 
in its 
inhomogeneous form,
namely having 
$\bff(\langle \bfk
\boldsymbol{\tau}
\rangle +W\langle \bfk \b{s}\rangle) $ in the exponential. 
Moreover, by shifting the center of the weight functions, one sees that Lemma \ref{lem3.15} still applies to the inhomogeneous form and then everything stays exactly the same with regard to the error terms.\par
To continue 
the proof of our lemma we perform 
the linear change of variables
$s_i\mapsto u_i$
and 
$\beta_i \mapsto \gamma_i$ 
given by 
$k_i(\tau_i+Ws_i)=B u_i$, $B^d \beta_i=\gamma_i$.
This leads to  the
following expression for the real density in our lemma,
\[
\frac{B^{n-Rd}}{W^n{\widetilde{\bfk}}}
\int_{\R^R}
\int_{\R^n}
e\left( 
\boldsymbol{\gamma}
\cdot
\bff
\left(\b{u}\right) 
\right)
\prod_{i=1}^n 
w
\bigg(
u_i
- 
\frac{\zeta_i}{2|\boldsymbol{\zeta}|}
\bigg)
\d\b{u}
\d\boldsymbol{\gamma}
,\]
which equals $\mathcal{J}_w(\bff,W)\widetilde{\bfk}^{-1} 
B^{n-Rd}
$. 
\end{proof}
The most noteworthy 
property of Lemma~\ref{lem:levofdi}
is related to the presence of $\widetilde{\bfk}^{-1}$ in the error term;
this allows to drastically
improve the level of distribution in the forthcoming 
applications.

\subsection{Using the Rosser--Iwaniec sieve}
\label{s:flst}
By~\eqref{eq:bwvfminor}
we have
the following 
whenever 
$z$ satisfies $z_0<z<B$,
\[
\sum_{\substack{\b{x} \in (\N\cap [-B,B])^n
\\
\bff(\bfx)=\b{0},
P^-(\widetilde{\bfx})>z}}\hspace{-0,1cm}1
\geq 
w_0^{-n}
\sum_{\substack{\b{x} \in \c{A}\\P^-(\widetilde{\bfx})>z}}
\prod_{i=1}^n
w\bigg( \frac{x_i}{B} -\frac{\zeta_i}{2|\boldsymbol{\zeta}|} \bigg)
.\] 
Let us now bring into play a lower bound sieve
sequence
$\lambda_k^-$ of
dimension $n$. 
Recall the definition of $\theta'$ in~\eqref{def:the}.
We shall make use of the terminology in~\cite[\S 11.8]{MR2647984};
in doing so we shall call the support of $\lambda^-$ by
$D:=B^\delta$, for some constant $\delta \in (0,\theta').$
Using $(1\ast \mu)(l)\geq (1\ast \lambda^-)(l)$
for $l=\gcd(P(z_0,z),\widetilde{\b{x}})$
yields
\[
\sum_{\substack{\b{x} \in (\N\cap [-B,B])^n
\\
\bff(\bfx)=\b{0},
P^-(\widetilde{\bfx})>z}}\hspace{-0,1cm}1
\geq 
w_0^{-n}
\sum_{\substack{k|P(z_0,z)\\k\leq B^\delta}}
\lambda_k^-
\sum_{\substack{\b{x} \in \c{A}\\ k| \widetilde{\bfx}}}\prod_{i=1}^n
w\bigg( \frac{x_i}{B} -\frac{\zeta_i}{2|\boldsymbol{\zeta}|} \bigg)
.\]
The proof of~\cite[Lem.8]{BF}
can be directly
adapted in the setting of arbitrary dimension,
thus providing  the equality of the inner sum over $\bfx$ to
\[
\mu(k)\sum_{\substack{\bf{k}\in \N^n\\p|\widetilde{\bfk} \Leftrightarrow p|k}}\mu(\bfk)
N_w(B;\bfk)
,\]
where here and throughout the rest of the paper we will use the notation
\[
\mu(\bfk):=\mu(k_1)\cdots \mu(k_n)
.\]
A moment's thought reveals that the succeeding function is multiplicative,
\begin{equation}
\label{def:ome}
g(k):=
\mathbf{1}_{(k,W)=1}(k)
 \mu(k) \sum_{\substack{\bfk\in \N^n\\ p|\widetilde{\bfk} \Leftrightarrow p|k}}\mu(\bfk)
\varpi(\b{k})
\widetilde{\bfk}^{-1}
,\end{equation}  
a notation which allows
to assort our conclusions so far
in the following form,
\[
\sum_{\substack{\b{x} \in (\N\cap [-B,B])^n\\\bff(\bfx)=\b{0},P^-(\widetilde{\bfx})>z}}\hspace{-0,1cm}1
\gg
B^{n-Rd}
\sum_{\substack{k|P(z_0,z)\\k\leq B^\delta}}
\lambda_k^-g(k)
+O\Bigg(B^{n+\eps}\sum_{k\leq B^\delta}
|\mu(k)|
\sum_{\substack{\b{k}\in \mathbb{N}^n
\\p|\widetilde{\bfk} \Leftrightarrow p|k}}
\frac{
|\mu(\bfk)|
}{\widetilde{\b{k}}}
E(B;\b{k}) 
\Bigg)
.\]
In bounding the error term we will be confronted with sums of the form
\[
b_k:=
|\mu(k)|
\sum_{\substack{\b{k}\in \mathbb{N}^n
\\p|\widetilde{\bfk} \Leftrightarrow p|k}}
\frac{
|\mu(\bfk)|
}{\widetilde{\b{k}}}
|\b{k}|^{\alpha_1}
\min\{k_i\}^{\alpha_2}
,\]
where $\alpha_i\geq 0 $.
Each $\b{k}$ making a contribution to $b_k$ 
satisfies 
$|\bfk|\leq k \leq \widetilde{\b{k}}$, therefore 
\[
b_k \ll
|\mu(k)| k^{\alpha_1+\alpha_2-1+\eps}
.\]
We deduce that for each $1\leq j \leq 3$,
the quantity
\[
B^{-\epsilon_{j,1}}
\sum_{k\leq B^\delta} k^{\epsilon_{j,2}+\epsilon_{j,3}-1}
\ll
B^{-\epsilon_{j,1}+\delta(\epsilon_{j,2}+\epsilon_{j,3})}
\]
becomes $\ll B^{-Rd-\epsilon'}$ for some $\epsilon'>0$
due to $\delta<\theta'$. 
Therefore,    
we can see that   
for each 
$\delta \in (0,\theta')$
and $\eps>0$ 
there exists $\eta=\eta(\eps,\delta)>0$
such that 
\[
B^{n+\eps}
\sum_{k\leq B^\delta}
|\mu(k)|
\sum_{\substack{\bf{k}\in \N^n\\p|\widetilde{\bfk} \Leftrightarrow p|k}}
\frac{
|\mu(\bfk)|
}{\widetilde{\b{k}}}
E(B;\b{k}) 
\ll
B^{n-Rd-\eta}
.\]
This leads to the conclusion that subject to the assertion 
\begin{equation}
\label{eq:lowerb}
\sum_{\substack{k|P(z_0,z)\\k\leq B^\delta}}
\lambda_k^-g(k)
\gg
(\log B)^{-n}
\end{equation}
we can establish Theorem~\ref{thm:rosiwa} due to
\[
\sum_{\substack{\b{x} \in (\N\cap [-B,B])^n\\\bff(\bfx)=\b{0},P^-(\widetilde{\bfx})>z}}\hspace{-0,1cm}1
\gg
\frac{B^{n-Rd}}{(\log B)^n}
.\]
To prove~\eqref{eq:lowerb}
we shall use~\cite[Th.11.12]{MR2647984}.
To this end,
for any polynomials
$h_i \in \Z[x_1,\ldots,x_n]$
we abbreviate 
\[
\sigma_p(p| \b{h}(\b{x}))
:=
\lim_{l\to+\infty}
p^{-(n-R)l}
\#\bigg\{1\leq \bfx \leq p^l
:
p^l|\b{f}(\b{x}),  p| \b{h}(\b{x})
\bigg\}
.\]
\begin{lemma}
\label{lem:local_interpret}
For each prime $p>z_0$ 
we have 
$
g(p)
\sigma_p
=
\sigma_p(p|x_1\cdots x_n)
$.
\end{lemma}
\begin{proof}
The definition~\eqref{def:ome}
furnishes
\[
g(p)
\sigma_p
=
\sum_{m=1}^n
\frac{(-1)^{m-1}}{p^m}
\sum_{\substack{
\bfj 
\in \{0,1\}^n
\\
|\bfj|_1=m
}}
\delta(\bfj)
,\]
thus,
letting
$N_{\bfj}(p^l)
:=
\#\big\{
1\leq x_1,\ldots,x_l\leq p^l:
\b{f}((p^{j_i}x_i))\equiv \b{0} \md{p^l} 
\big\}$,
we conclude that
\begin{equation}
\label{eq:prf} 
g(p)
\lim_{l\to+\infty}
\frac{
N_{\boldsymbol{\b{0}}}(p^l)
}
{p^{(n-R)l}}
=
\lim_{l\to+\infty}
\sum_{m=1}^n
\frac{(-1)^{m-1}}{p^m}
\sum_{\substack{
\bfj
\in \{0,1\}^n
\\
|\bfj|_1=m
}}
\frac{
N_{\bfj}(p^l)
}
{p^{(n-R)l}}
.\end{equation} 
Obviously, if 
$j_i=1$
and
$y_i\equiv x_i \md{p^{l-1}}$
then 
$p^{j_i}y_i\equiv p^{j_i}x_i \md{p^{l}}$.
Therefore we may split the interval $[1,p^l]$ 
into $p$ subintervals of length $p^{l-1}$
to obtain
\[
N_{\bfj}(p^l)
=p^{|\bfj|_1}
\#\Big\{
j_i=1\Rightarrow
1\leq x_i \leq p^{l-1},
j_i=0\Rightarrow
1\leq x_i \leq p^{l}
:
\b{f}((p^{j_i}x_i))\equiv \b{0} \md{p^l} 
\Big\}
.\]
One can see that this entity equals
$
\#\Big\{ 
\b{x} \leq p^{l}
:
\b{f}(\b{x})\equiv \b{0} \md{p^l},
j_i=1\Rightarrow
p|x_i
\Big\}
$, hence, combining this with~\eqref{eq:prf} 
yields the desired result.
\end{proof}  
\begin{lemma}
\label{lem:cookmag}
There exists $\eps_0\in (0,1)$ such that one has
$$g(p)= \frac{n}{p}
+
O(p^{-1-\eps_0}).$$
\end{lemma}
\begin{proof}
For a prime $p$ and $t\in \N$ let
$M(p^t):=\sharp\{1\leq \bfx \leq p^t: p^t|\bff(\bfx)
,\ p\nmid x_1\cdot\ldots\cdot x_n\}$.
Then Lemmas $11$-$12$  in~\cite{CM} imply that there exists a positive $\eps_0>0$ such that
$$\left(1-\frac{1}{p}\right)^{-n} \lim_{t\rightarrow \infty} p^{-t(n-R)}M(p^t) = 1+O(p^{-1-\eps_0}).$$
We observe that
$\lim_{t\rightarrow \infty} p^{-t(n-R)}M(p^t) 
= \sig_p - \sig_p(p|x_1\ldots x_n)$,
thus Lemma~\ref{lem:local_interpret}
reveals that 
\begin{equation*}
\begin{split}
g(p)
&= \frac{\sig_p(p|x_1\ldots x_n)}{\sig_p} \\&= 1 - \sig_p^{-1}\lim_{t\rightarrow \infty} p^{-t(n-R)}M(p^t) \\&= 1-\left(1-\frac{1}{p}\right)^n\frac{1}{\sig_p}+O(\sig_p^{-1}p^{-1-\eps_0}).
\end{split}
\end{equation*}
The work of Birch \cite{Bir62} establishes the existence of a positive $\epsilon_1$ such that
$\sig_p=1+O(p^{-1-\eps_1}).$
This is sufficient for our lemma. 
\end{proof}
Enlarging $z_0$
if necessary, 
ensures that
for all primes $p$ we have  
\[
0\leq g(p) <1
\
\
\text{and}
\
\
g(p)\leq \frac{n}{p}+O(p^{-1-\epsilon_0})
.\]
This means that one can take $
\kappa=n$ in~\cite[Eq.(11.129)]{MR2647984},
hence our sieve problem
has dimension  $n$.
By~\cite[Th.17.2,Prop.17.3]{MR2458547},
the
sieving limit
$\beta$
fulfils 
$
\beta
\leq 
3.75 
n
$,
thus~\cite[Th.11.12]{MR2647984},
in combination with Lemma~\ref{lem:cookmag},
guarantees the veracity of~\eqref{eq:lowerb}
under the condition
\[
\frac{\log D}{\log z}>
3.75 n
.\]
This concludes the proof of Theorem~\ref{thm:rosiwa}. 
\subsection{Using the weighted sieve}
\label{s:weightedwww}
In the last section we saw that sieving out small prime divisors of $x_1\cdots x_n$
for integer zeros of $\bff(\b{x})=\b{0}$
gives rise to a sieve of dimension $n$.
When the dimension of the sieve increases
then the weighted sieve gives better results for the number of prime divisors in our sequence.
We would like to use the weighted sieve in the form given in the Cambridge Tract of
Diamond and Halberstam~\cite[Th.11.1]{MR2458547},
however we shall need a more flexible version of their work;
one that allows the use of smooth weights.
This will follow from a generalisation of the weighted sieve
that will be given in~\S\ref{s:wssw}.
This generalisation permits
the use of any suitable non-negative function
rather than just a smooth weight
as well as sieving in multisets.
\subsubsection{The weighted sieve with smooth weights}
\label{s:wssw}
We assume that $\c{M}$ is any set equipped with two functions
$
\pi:\c{M}\to\N,
h:\c{M} \to \R 
$
such that 
\begin{equation}
\label{cond:hhh}
h(\c{M})\subset [0,1],
h \neq 0,
\#\c{M}<\infty
.\end{equation}
For convenience of presentation we shall prefer the notation 
$\overline{m}=\pi(m)$. We also assume that there exists a set of primes 
$\c{P}$, a constant
$X\in \R_{>0}$ and a multiplicative function $\omega:\N\to\R_{\geq 0}$ 
such that, when letting
\[r_{\!\c{M},h}(k)
:=
\sum_{\substack{m \in \c{M}\\b|\overline{m}}}h(m)
-\frac{\omega(b)}{b}X
,
\
(b \in \N),
\]
there exist constants $\tau \in (0,1]$, 
$\kappa \in \N$,
$A_1\geq 1$ and $A_2\geq 1$
such that
\begin{equation}
\label{def:tau} 
\sum_{1\leq b \leq X^\tau (\log X)^{-A_1}}
\mu(b)^2 
4^{\nu(b)}
|r_{\!\c{M},h}(b)|
\leq 
A_2 \frac{X}{(\log X)^{\kappa+1}}
,\end{equation}
where
the function $\omega$ 
enjoys the following properties for some constants $\kappa\geq 1,A>1$,
\begin{equation}
\label{cond:omegg}
0\leq \omega(p)<p \ (p \in \c{P}),
\
\omega(p)=0
\
(p\notin \c{P}) 
\end{equation}
\begin{equation}
\label{cond:omegd} 
\prod_{w_1\leq p <w}
\Big(1-\frac{\omega(p)}{p}\Big)^{-1}
\leq 
\Big(\frac{\log w}{\log w_1}\Big)^\kappa
\Big(1+\frac{A}{\log w_1}\Big),
\
2\leq w_1 < w.
\end{equation}
We furthermore demand that 
\begin{equation}
\label{cond:onlym}
m \in \c{M},p| \overline{m}
\Rightarrow
p \in \c{P} 
,\end{equation}
and that  
that there exists a constant $\mu_0>0$ such that 
\begin{equation}
\label{cond:suppo} 
\max\{
|\overline{m}|: m \in \c{M}\}
\leq X^{\tau \mu_0}
.\end{equation}
Lastly, we shall say that the property $\mathbf{Q}(u,v)$ holds for two real positive
numbers
$u<v$ if
\begin{equation}
\label{cond:Q}
\mathbf{Q}(u,v):
\
\
\
\
\sum_{\substack{X^{1/v}\leq p \leq X^{1/u} \\ p \in \c{P}}}
\
\sum_{\substack{m \in \c{M}\\p^2|\overline{m}}}h(m)
\ll
\frac{X}{\log X} \prod_{\substack{ p\in \c{P}\\ p<X^{1/v}}}
\Big(1-\frac{\omega(p)}{p}\Big)
.\end{equation}

Before stating the main theorem in this section
recall the definition of $f=f_\kappa,F=F_\kappa$ and $\beta_\kappa$ in~\cite[Th.6.1]{MR2458547}
through certain differential equations.
The inequality
$
\beta_\kappa
<
\nu_\kappa
$
is proved
for $\kappa\geq 200$
in~\cite[Th.2]{MR1485427};
here $\nu_\kappa$ 
is the Ankeni-Onishi sieving limit\cite{MR0167467}
that satisfies
$
\nu_\kappa\sim c \kappa
$ as $\kappa \to+\infty$, where
\[
c=\frac{2}{e\log 2}
\exp\Bigg(
\int_0^2
\frac{e^u-1}{u}\mathrm{d}u\Bigg)
=2.445\ldots
\
.\]
In particular
there exists an absolute positive constant $c_0$
such that
$\beta_\kappa
\leq c_0 \kappa$
for all 
$\kappa\geq 1$.
\begin{theorem}
[Diamond--Halberstam--Richert]
\label{thm:weighted0}
Assume that $\kappa\geq 1,\c{M},X,\omega,\mu_0$ are as above,
that each one of the conditions~\eqref{cond:hhh}-\eqref{cond:suppo} holds,
that $r$ is a natural number satisfying
$r>N(u,v;\kappa,\mu_0,\tau)$,
where
\[
N(u,v;\kappa,\mu_0,\tau)
:=
\tau \mu_0 u-1+\frac{\kappa}{f_\kappa(\tau v)}
\int_u^v
F_\kappa\Big(v\Big(\tau-\frac{1}{s}\Big)\Big)
\Big(1-\frac{u}{s}\Big)
\frac{\mathrm{d}s}{s}
\]
and $u,v$ satisfy 
$\mathbf{Q(u,v)},
\tau v>\beta_\kappa$ as well as $1/\tau<u<v$.
Then
we have 
\[
\#
\Big\{
m \in \c{M},
P^-(\overline{m})\geq X^{1/v},
\Omega(\overline{m})\leq r
\Big\}
\gg
X \prod_{\substack{ p\in \c{P}\\ p<X^{1/v}}}
\Big(1-\frac{\omega(p)}{p}\Big)
.\]
\end{theorem}
\begin{proof}
The proof is merely a careful recast
of the proof of Theorem $11.1$ in~\cite[\S 11]{MR2458547}.
In place of the function defined in~\cite[Eq.(11.6)]{MR2458547}
we shall use the following function that combines 
the classical weights related to
the weighted sieve in addition to the new weight $h$,
\[
W_h(\c{M},\c{P},z,y,\lambda)
:=\sum_{\substack{
m \in \c{M}\\ \gcd(\overline{m},P(z))=1}}
h(m)
\Bigg\{
\lambda-
\sum_{\substack{p \in \c{P},p|\overline{m}\\ z\leq p < y}}
\Big(1-\frac{\log p}{\log y}\Big)
\Bigg\}
,\]
where 
$P(z):=\prod\{p:p\in \c{P},p<z\}$.
A statement analogous to~\cite[Eq.(11.9)]{MR2458547} can be verified once
the entities 
$S(\c{A},\c{P},X^{1/v})$
and
$S(\c{A}_p,\c{P},X^{1/v})$ are replaced
by
\[
\sum_{\substack{m \in \c{M}\\
\gcd(\overline{m},P(X^{1/v}))=1
}}
h(m)
 \ \ \ \text{ and } \ 
\sum_{\substack{m \in \c{M}, p|\overline{m}
\\
\gcd(\overline{m},P(X^{1/v}))=1
}}
h(m)
\]
respectively.
The rest of the arguments in~\cite[\S 11.2]{MR2458547} are carried automatically to our setting
since, once the level of distribution result~\eqref{def:tau} is applied, all information
regarding $\c{M}$ and $h$ is absorbed into  $X$.
The only point of departure
is the use of various sieve estimates from previous chapters 
of the book. These sieve estimates boil down to the use of the Fundamental lemma of sieve theory 
and the Selberg sieve, both of which can be adapted to our setting.
This is due to the non-negativity of the function $h$, which allows various combinatorial inequalities
to be adapted once multiplied by $h$.
One example of this  
is in the case of an upper bound sieve, say $\lambda^+$:
opening up the convolution in the right side of  
$(1\ast \mu) \leq (1\ast \lambda^+)$
gives
\[
\sum_{\substack{m \in \c{M}\\ \gcd(\overline{m},P(z))=1}}
h(m)
\leq 
\sum_{k|P(z)} \lambda_k^+
\sum_{\substack{m \in \c{M}\\ k|\overline{m}}}h(m)
,\]
and one can now use~\eqref{def:tau} to absorb $\c{M}$ and $h$ in $X$
for
the rest of the argument.

For the proof of the present 
theorem it remains to adapt the arguments in~\cite[\S 11.3]{MR2458547}.
First, the contribution towards $\sum_{m}h(m)$ 
of those 
$m \in \c{M}$
such that $\overline{m}$ is divisible by the square of a prime $p\in \c{P}$ in the range
$X^{1/v}\leq p \leq  X^{1/u}$
can be safely ignored 
due to condition~\eqref{cond:Q}.
An inspection of~\cite[\S 11]{MR2458547}
reveals
that condition $\mathbf{Q_0}$
in~\cite[Eq.(11.2)]{MR2458547} is used in the proof of~\cite[Th.11.1]{MR2458547}
only
to deal with this particular sum over primes in $\c{P} \cap [X^{1/v},X^{1/u}]$.
We are thus free to focus
our attention exclusively
on the contribution of the
elements of the set
\[\c{M}':=\big\{m'\in\c{M}:
\text{ there is no prime }
 p \in \c{P} \cap [X^{1/v},X^{1/u}]
\text{ such that }
p^2\mid \overline{m}'\big\}
.\]  
The last inequality in~\cite[p.g.140]{MR2458547}
becomes
\[
\sum_{\substack{X^{1/v}\leq p <X^{1/u}\\p\in \c{P},p|\overline{m}'}}
\Big(1-\frac{u \log p}{\log X}
\Big)
\geq
\Omega(\overline{m}')
-\frac{u \log |\overline{m}'|}{\log X}
,\] which, when multiplied by $h(m')$, gives, as in~\cite[p.g.141]{MR2458547},
\[
W_h(\c{M}',\c{P},z,y,\lambda)
\leq
(r+1)\sum_{\substack{m'\in \c{M}', \Omega(\overline{m}')\leq r \\ \gcd(\overline{m}',P(X^{1/v}))=1}}h(m')
\]
for the choice of $\lambda$ and $r$ made in~\cite[p.g.141]{MR2458547}.
The property $h(\c{M})\subset [0,1]$ 
shows that
\begin{align*}
\#\{m\in \c{M}:P^-(\overline{m})\geq X^{1/v}, \Omega(\overline{m})\leq r\}
&\geq
\#\{m'\in \c{M}': P^-(\overline{m}')\geq X^{1/v},
\Omega(\overline{m}')\leq r\}
\\
&\geq
\sum_{\substack{m'\in \c{M}, \Omega(\overline{m}')\leq r \\ \gcd(\overline{m}',P(X^{1/v}))=1}}h(m')
\geq
\frac{1}{r+1}
W_h(\c{M}',\c{P},z,y,\lambda),
\end{align*}
which allows 
the rest of the proof of~\cite[Th.11.1]{MR2458547}
to be adapted to our case.
Finally,
the choice
of the constants $v$ and $r$
given in our theorem 
is borrowed from the inequalities succeeding~\cite[Eq.(11.22)]{MR2458547}.
\end{proof}
\begin{remark}
\label{rem:recov}
The setting of
Theorem~\ref{thm:weighted0}
includes that of~\cite[Th.1.1]{MR2458547};
indeed, one can choose
$
(\c{M},\pi,h)
=(\c{A}, \mathrm{id}, 1)
$.
\end{remark} 
\begin{remark}
\label{rem:uv12}
In most cases
it is easy to verify 
$\mathbf{Q}(u,v)$ for all $u,v>0$,
however this is not the case for the problem of prime factors of 
$x_1\cdots x_n$ for
integer solutions $\b{x}=(x_1,\ldots,x_n)$ of general Diophantine equations,
since, as explained in \S\ref{s:localden}, quite often
a prime could divide two coordinates of $\b{x}$.
\end{remark}  
\begin{remark}
\label{rem:betterr}
A table of estimates for $\beta_\kappa$ for $1\leq \kappa\leq 10$ is given in~\cite[p.g.227]{MR2458547}.
Furthermore,~\cite[Eq.(11.21)]{MR2458547}
contains estimates for $r$ that are slightly weaker but simpler than that of~\cite[Th.11.1]{MR2458547}.
For example, the choice $\xi=\beta_\kappa$ in~\cite[Eq.(11.21)]{MR2458547}
shows that, as long as 
$\mathbf{Q}\big(\frac{2\beta_\kappa-1}{\tau \beta_\kappa},\frac{2\beta_\kappa-1}{\tau}\big)$ holds,
then the conclusion of Theorem~\ref{thm:weighted0} remains valid with 
$v=(2\beta_\kappa-1)/\tau$ and 
for all natural numbers $r$ satisfying
\begin{equation}
\label{eq:lower1r}
r\geq \mu_0-1+(\mu_0-\kappa)(1-1/\beta_\kappa)+(\kappa+1)\log \beta_\kappa
.\end{equation}
In
fact~\cite[Eq.(11.21)]{MR2458547}
with $\xi=\beta_\kappa$ shows that 
if 
$\mathbf{Q}(u,v)$ holds for some
$u>1/\tau$ and any $v>u$, then 
letting
\[
v':=\frac{\beta_\kappa-1}{\tau-1/u}
\]
we deduce that  
the conclusion of Theorem~\ref{thm:weighted0} still holds with
any $r$ satisfying
\begin{equation}
\label{eq:lower2r}
r\geq \tau \mu_0 u-1
+
\Big(\kappa+\frac{u}{v'}\beta_\kappa\Big)
\log \frac{v'}{u}
-\kappa
\Big(1-\frac{u}{v'}\Big)
.\end{equation}
\end{remark} 
To prove Theorem~\ref{thm:weighted}
we take 
\[
\c{M}:=\Big\{\b{x}\in \N^n:
\b{f}(\b{x})=\b{0},\bfx\equiv \bfy\md{W},|\b{x}|\leq B
\Big\},
\pi(\bfx):=\widetilde{\bfx},
\]
and
we let 
\[
h(\bfx):=\prod_{i=1}^n 
w\bigg( \frac{x_i}{B} -\frac{\zeta_i}{2|\boldsymbol{\zeta}|} \bigg) 
.\]
Then for $\c{P}$ being the set of all primes $p>z_0$,
$g$ as in~\eqref{def:ome},
$\theta'$ as in~\eqref{def:the}
and any $0<\epsilon < \theta'$
we can verify all
conditions~\eqref{def:tau}-\eqref{cond:suppo}
with 
\[
X:=\mathcal{J}_w(\bff,W)\mathfrak{S}(\bff,W)
B^{n-Rd},
\omega(b):=b g(b),
\kappa:=n,
\tau:=
\theta'-\epsilon
,
\mu_0=\frac{n}{n-Rd}\frac{1+\epsilon}{\theta'-\epsilon}
\]
with an argument that is identical to that in~\S\ref{s:flst}.
It remains to check condition $\mathbf{Q}(u,v)$
and for this we note that in our setting,
the sum in~\eqref{cond:Q}
is at most 
\[
\sum_{X^{1/v}< p \leq X^{1/u}}
\sum_{\substack{\bfk \in \N^n\\ \widetilde{\bfk}=p^2}}
\sum_{\substack{\b{x} \in \c{A}\\ k_i|x_i}}
\prod_{i=1}^n w\bigg( \frac{x_i}{B} -\frac{\zeta_i}{2|\boldsymbol{\zeta}|} \bigg) .\]
Invoking Lemma~\ref{lem:levofdi} we see that, if
$
u>
2(n-Rd)
\rho$,
where $\rho$ is defined in~\eqref{def:rrr},
this is 
\[
\ll
B^{n-Rd}
\Bigg(
\sum_{X^{1/v}\leq p \leq X^{1/u}}
p^{-2}
\sum_{\substack{\bfk \in \N^n\\ \widetilde{\bfk}=p^2}}
\varpi(\bfk)
\Bigg)
+B^{n+\eps}
\Bigg(
\sum_{X^{1/v} <  p \leq X^{1/u}}p^{-2}
\sum_{\substack{\bfk \in \N^n\\ \widetilde{\bfk}=p^2}}
E(B;\b{k})
\Bigg)
.\]
Assuming
$
\max\{
(d^2-1)R^2
2^{d-1}
,
(d-1)R^2
2^{d-1}
+2(R+1)
\}
<
\mathfrak{B}(\b{f})
$,
we obtain via 
Corollary~\ref{cor2}
that the first sum over $\bfk$ above is $\ll 1$, thus, when $v> 0$,
the first term contributes
\[
\ll
B^{n-Rd-\frac{(n-Rd)}{v}}
\ll 
\frac{ B^{n-Rd}}{ (\log B)^{n}}
.\]
It remains to verify that there exists $\epsilon'>0$ such that 
\begin{equation}
\label{eq:holds2}
\sum_{X^{1/v} <  p \leq X^{1/u}}p^{-2}
\sum_{\substack{\bfk \in \N^n\\ \widetilde{\bfk}=p^2}}
E(B;\b{k}) 
\ll B^{-\epsilon'-Rd}
.\end{equation}
For this we note that 
each $\epsilon_{i,2}$ is at least $\frac{1}{2}$
owing to $K\geq \max\{Rd,R^2(d-1)\}$ 
and $K\geq 1$.
Thus the error term above becomes
\[
\ll
\sum_{i=1}^3 B^{-\epsilon_{i,1}}
\sum_{X^{1/v} <  p \leq X^{1/u}}p^{-2+2\epsilon_{i,2}}
\ll
\sum_{i=1}^3 B^{-\epsilon_{i,1}+\frac{(n-Rd)}{u}(2\epsilon_{i,2}-1)}
.\]
Therefore, if 
\[
u>\max\Bigg\{\frac{(n-Rd)
(2\epsilon_{i,2}-1)}{\epsilon_{i,1}-Rd}:1\leq i \leq 3
\Bigg\}
\]
then~\eqref{eq:holds2} holds.
Now define 
$u_0:=(1+\epsilon)
\max
\big
\{u_1,1/(\theta'-\epsilon),
2(n-Rd)\rho
\big
\}$, where
\[
u_1:=
\max\Bigg\{\frac{(n-Rd)
(2\epsilon_{i,2}-1)}{\epsilon_{i,1}-Rd}:1\leq i \leq 3
\Bigg\}.\] 
Then
applying~\eqref{eq:lower2r}
with $u=u_0$ and
$v':=(n c_n-1)/(\theta'-\epsilon-1/u_0)$,
allows to take 
\[
r\geq 
\frac{n}{n-Rd}
(1+\epsilon) 
 u_0-1
+
n\Big(1+\frac{u_0}{v'}c_n\Big)
\log \frac{v'}{u_0}
-n
\Big(1-\frac{u_0}{v'}\Big)
,\]  
where 
$c_n:=\beta_n/n$ satisfies
$\lim_{n\to+\infty}c_n=2.44\ldots$.
Letting 
$\epsilon>0$ be
arbitrarily close to zero 
concludes
the proof of
the lower bound claimed in Theorem~\ref{thm:weighted}.
This is because 
the quantities 
$u'',\widehat{u},\widehat{v}$ introduced in~\eqref{eq:epikoskafes}
and~\eqref{eq:epikoskafes2}
are such that for fixed $\b{f},n,d,R$ we have 
\[
\lim_{\epsilon\to 0}
(u_1,u_0,v)=
(u'',\widehat{u},\widehat{v})
.\]

To complete the proof of 
Theorem~\ref{thm:weighted}
it remains to verify the estimates regarding 
$\widehat{v}$ and
$
r_0
$,
where $r_0$ is defined in~\eqref{eq:epikoskafes3}.
It is easy to
see that $u''/(n-Rd)$ is a function of $K$ that is bounded away from $0$ and $+\infty$,
while a similar remark applies to $\rho$ and $\theta'$. 
This implies that $\widehat{u}\ll_{d,R} n$
and
noting that 
$\widehat{u}<\widehat{v}$,
one has
\[
r_0\ll_{d,R}
\widehat{u}+n
\log \frac{\widehat{v}}{\widehat{u}}
\ll_{d,R}
n(1+\log \frac{\widehat{v}}{\widehat{u}})
,\] 
where the implied constant is independent of $K$ and $n$.
The identity 
$n c_n-1=\tau \widehat{v}-\widehat{v}/\widehat{u}$ 
shows that 
\[
\widehat{v}/\widehat{u}
\ll n+\widehat{v}
\ll
n+\frac{n}{\theta'-1/\widehat{u}}
\ll
n,\]
 therefore 
$
r_0=O_{d,R}(n \log n)
$,
with an implied constant depending at most on $d$ and $R$.

\section{Multidimensional vector sieve}
\label{s:multidim}
The next lemma 
constitutes 
a generalisation of the vector sieve
of Br\"{u}dern and Fouvry~\cite{BF}
to arbitrarily many variables.
\begin{lemma}[Multidimensional vector sieve]
\label{lem:multi}
Let $n\in \N$ and assume that we are given 
$2$ sequences 
$\lambda_i^-,\lambda_i^+, (i=1,\ldots,n)$
such that for each $m \in \N$ and $1\leq i \leq n$ 
we have 
\begin{equation}
\label{eq:uplo}
(1\ast \lambda_i^-)(m)
\leq 
(1\ast \mu)(m)
\leq 
(1\ast \lambda_i^+)(m)
.\end{equation}
Then the following inequality holds for each  
$\b{m}\in \N^n$,
\[
\prod_{i=1}^n
(1\ast \mu)(m_i)
\geq  
\sum_{i=1}^n
(1\ast \lambda_i^-)(m_i)
\prod_{\substack{1\leq j \leq n\\j \neq i }}
(1\ast \lambda_j^+)(m_j)
-(n-1)
\prod_{i=1}^n
(1\ast \lambda_i^+)(m_i)
\]
\end{lemma}
\begin{proof}
In light
of~\eqref{eq:uplo}
it is sufficient to verify 
\begin{equation}
\label{eq:uplo1}
\prod_{i=1}^n
(1\ast \mu)(m_i)
\geq 
-(n-1)
\prod_{i=1}^n
(1\ast \lambda_i^+)(m_i)
+
\sum_{i=1}^n
(1\ast \mu)(m_i)
\prod_{\substack{1\leq j \leq n\\j \neq i }}
(1\ast \lambda_j^+)(m_j)
.\end{equation}
If
$m_i=1$ for all $i=1,\ldots,n$
then $(1\ast\lambda_i^+)(m_i)\geq 1$,
thus
the entities
$x_i:=1/(1\ast\lambda_i^+)(m_i)$
fulfill 
$0<x_i\leq 1$.
The inequality~\eqref{eq:uplo1}
becomes 
$
x_1\cdots x_n \geq -n+1+ (x_1+\cdots+x_n)$.
Letting 
$A_i=1-x_i$ the last inequality 
becomes 
$
(1-A_1)\cdots (1-A_n)
\geq 1-(A_1+\cdots+A_n)$,
which
is the
Weierstrass product inequality, see~\cite[Eq.(1)]{MR0265536}.
In the remaining case where there
exists $i$ with $m_i\neq 1$
we can assume that 
$(1\ast\lambda_i^+)(m_i)\neq 0$
for each such $i$,
for otherwise both sides of~\eqref{eq:uplo1}
vanish.
We
may now introduce
for each $1\leq i \leq n$ the variables
$x_i:=1/(1\ast\lambda_i^+)(m_i)$;
then~\eqref{eq:uplo1}
becomes
\[
n-1\geq
\sum_{\substack{1\leq i \leq n\\ m_i=1}}x_i.\]
The proof is concluded upon observing that 
the condition
$m_i=1$ implies 
$x_i\leq 1$.
\end{proof}

Our aim  
now becomes to prove a version of the Fundamental Lemma of sieve theory 
in the context of prime divisors of coordinates of integer zeros in varieties.
The exact form is given in Proposition~\ref{prop:presiev} and the rest of this section is devoted to its proof.
The quantity under consideration is the weighted
density of vectors $\bfx \in \c{A}$ with 
$|\bfx|\leq B$
such that  $\widetilde{\b{x}}$ does not have 
prime divisors in the range $p\leq z_1$
for any $z_1$ with 
$z_0<z_1\leq B$.
We prefer to keep the choice of $z_1$
unspecified in this section
and we shall only need the value
$z_1= (\log B)^A$ for $A>0$ independent of $B$
in~\S\ref{s:mainthm}.

For $\b{k} \in \N^n$ 
and
$y_1,y_2 \in \R$ with $y_1<y_2$
we
define 
\[
\mu(\b{k}):=\prod_{i=1}^n \mu(k_i)
\ \text{ and } \
P(y_1,y_2):=\prod_{y_1<p \leq y_2}p
.\] 
For a smooth  function $w:\R\to \R_{\geq 0}$ 
that is as in subsection~\ref{s:bwv1056largo},
any $z_1>z_0$  
and any $\b{l} \in \N^n$ we let  
\begin{equation}
\label{def:www}
G(B,z_1;\b{l})
:=
\sum_{\substack{
\b{x}\in \c{A}\!, \ 
l_i|x_i, 
\\
p|x_1\cdots x_n \Rightarrow p>z_1
}}
\
\prod_{i=1}^n
w\bigg( \frac{x_i}{B} -\frac{\zeta_i}{2|\boldsymbol{\zeta}|} \bigg) 
.\end{equation} 
We are interested in 
estimating 
$G(B,z_1;\b{l})$
whenever
$\b{l}\in \N^n$ fulfills  $l_i|P(z_1,z)$,
where $z$ is any constant satisfying $z>z_1$.
This is analogous to~\cite[Prop.p.g.83]{BF}
and we shall also begin by proving the upper bound.
We shall use the upper and lower bound sieves,
$\lambda^+$ and $\lambda^-$, 
as defined
at the bottom of~\cite[p.g.84]{BF}.
Assume that $\lambda^+$ is an upper bound sieve 
supported in $[1,D_1]$
and  
note that 
the condition $\b{x}\equiv \b{y}\md{W}$
ensures that $p\nmid \widetilde{\bfx}$ 
for all $p\leq z_0$.
Recalling definition~\eqref{def:nw}
we see that whenever  
$l_i|P(z_1,z)$
then 
\[
G(B,z_1;\b{l})
\leq 
\sum_{\substack{\b{k} \in \N^n
\\ k_i |P(z_0,z_1)}}
N_w(B;(k_1l_1,\ldots,k_nl_n))
\prod_{i=1}^n
\lambda^+_{k_i}
.\]
Note that all
$\b{k}$ and $\b{l}$
above
must satisfy
\[
\gcd(\widetilde{\bfk},\widetilde{\bfl})=1=
\gcd\Big(\widetilde{\bfk} \ \widetilde{\bfl},\prod_{p\leq z_0}p\Big),
\mu(k_i)^2=1=\mu(l_i)^2
.\]
Recall definition~\eqref{def:rrr}
and assume
that 
\begin{equation}
\label{eq:levd9}
|\b{l}|
\leq \frac{B^{1/\rho}}{D_1\log B} 
.\end{equation}
Then Lemma~\ref{lem:levofdi}
shows that  if $K>R(R+1)(d-1)$ and~\eqref{eq:levd9}
holds then  
\[ 
N_w(B;
(k_1l_1,\ldots,k_nl_n)
)
=
\frac{\varpi(\bfk)}{\widetilde{\bfk}}
X_\bfl
+O\Big(
\frac{B^{n+\epsilon}
}{
\widetilde{\b{l}}
}
\frac{E(B;
(k_1l_1,\ldots,k_nl_n)
)}
{\widetilde{\b{k}}}
\Big)
,\]
where
\[
X_\bfl:=
\mathfrak{S}(\bff)
\mathcal{J}_w(\bff,W)
\frac{\varpi(\bfl)}{{\widetilde{\b{l}}}}
B^{n-Rd}
.\] 
We may now set 
\[
\Sig(D_1,z_1)=\sum_{\substack{\b{k} \in \N^n
\\ k_i |P(z_0,z_1)}}
\frac{\varpi(\bfk)}{\widetilde{\bfk}}\prod_{i=1}^n\lambda^+_{k_i}
\]
to obtain
\begin{equation}
\label{eq:bigsigma} 
G(B,z_1;\b{l})
\leq \Sig(D_1,z_1)
X_\bfl +O
\Bigg
( 
\frac{B^{n+\epsilon}}{\widetilde{\b{l}}} 
\sum_{\substack{|\bfk| \leq D_1\\
p|\widetilde{\bfk} \Rightarrow z_0<p \leq z_1}}
\frac{\mu(\b{k})^2}{\widetilde{\b{k}}}
E(B;
(k_1l_1,\ldots,k_nl_n)
)
\Bigg).
\end{equation}
\subsection{Bounds for $\varpi$.}
\label{s:bvarpi}
One has to be careful when 
adapting the approach \cite{BF} 
to homogeneous equations.
The reason is that in the case of
Lagrange's
equation there exists
a multiplicative function $\widetilde{\varpi}$ satisfying
$$
\varpi(\b{m})\leq \prod_{i=1}^n \widetilde{\varpi}(m_i)
$$ 
and such that for all large primes $p$ one has $\widetilde{\varpi}(p)\leq 2$, see~\cite[Lem.12,part(iii)]{BF}.
It is easy to see that bounds of this quality
fail to hold 
rather 
spectacularly
for systems of 
forms $\bff=\b{0}$ as in Theorem~\ref{thm:mainvector}.
Indeed,
\begin{equation*}
\begin{split}
\varpi(p,\ldots, p)&=\sig_p^{-1}\lim_{l\rightarrow\infty}p^{-l(n-R)}\sharp\Big\{\bfx\mod{p^l}: \bff(p\bfx)\equiv 0 \mod{p^l}\Big\}\\
&= p^{Rd}\sig_p^{-1}\lim_{l\rightarrow\infty}p^{-(l-d)(n-R)}\sharp\Big\{\bfx\mod{p^{l-d}}: \bff(\bfx)\equiv 0 \mod{p^{l-d}}\Big\}\\
&= p^{Rd}.
\end{split}
\end{equation*} 
To confront this issue
our first task is to control the contribution 
towards $\Sigma(D_1,z_1)$
of integer vectors $\bfk$ 
such that there exists $i<j$
with $k_{ij}:=\gcd(k_i,k_j)$ 
attaining a large value.
Define
\[
\Sig^*(D_1,z_1)=\sum_{\substack{\b{k} \in \N^n\\ k_i |P(z_0,z_1)\\ \max k_{ij}\leq \Del}}
\frac{\varpi(\bfk)}{\widetilde{\bfk}}\prod_{i=1}^n\lambda^+_{k_i}
\]
and
recall the definition of $\Upsilon$ in~\eqref{def:theY}.
\begin{lemma}\label{lem3}
Assuming  
$
\max\!\big\{
(d-1)R^2
2^{d-1}
+(R+1)
(\Upsilon+1),
(d^2-1)R^2
2^{d-1}
\big\}
\!<\!\mathfrak{B}(\b{f})
$,
one has   
$$\Sig(D_1,z_1)-\Sig^*(D_1,z_1)\ll \Del^{-1+\eps} (\log z_1)^n.$$
\end{lemma}

\begin{proof} 
The quantity under investigation
is $\ll \sum_{1\leq l_1<l_2\leq n}\calE(l_1,l_2)$, where 
$$\calE(l_1,l_2):=\sum_{\del >\Del} \mu(\del)^2
\hspace{-0,3cm}
\sum_{\substack{k_i|P(z_0,z_1)\\ \del |k_{l_1}, \del |k_{l_2}}} 
\frac{\varpi(\bfk)}{\widetilde{\bfk}}
.$$
We may now use the multiplicative properties of $\varpi$ to deduce that
\begin{equation*} 
\calE(l_1,l_2)
\ll \sum_{\del> \Del}
\Bigg(
\prod_{\substack{z_0<p\leq z_1\\ p|\del}}
\sum_{\substack{\bfj\in \{0,1\}^n\\ j_{l_1}=j_{l_2}=1}}
\frac{\varpi(p^{\bfj})}{p^{|\bfj|_1}}
\Bigg)
\Bigg(
\prod_{\substack{z_0<p\leq z_1\\ p\nmid\del}}
\sum_{\bfj\in \{0,1\}^n}
\frac{\varpi(p^{\bfj})}{p^{|\bfj|_1}}
\Bigg)
.\end{equation*} 
Fix $\eta \in (0,1/4)$
and
let us 
denote 
$s_0:=\Upsilon+1+\eta$.
By Corollary~\ref{corome} we obtain 
\[
\sum_{|\bfj|_1\geq s_0}
\frac{\varpi(p^{\bfj})}{p^{|\bfj|_1}}
\ll
p^\Upsilon
\sum_{s_0\leq s \leq n} {{n}\choose{s}} p^{-s}
\ll 
p^{-1-\eta}
.\]
The assumptions of our 
lemma
allow us to apply Corollary~\ref{cor2}
whenever $|\bfj|_1\leq s_0$.
Thus it supplies us with some $\lambda>0$
such that  
$\varpi(p^{\bfj})=1+O(p^{-\lambda})$,
which yields 
$$\sum_{\bfj\in \{0,1\}^n}
\frac{\varpi(p^{\bfj})}{p^{|\bfj|_1}} = 1+\frac{n}{p} +O(p^{-1-\eps})
\ \text{ and }  \
\sum_{\substack{\bfj\in \{0,1\}^n\\ j_{l_1}=j_{l_2}=1}}
\frac{\varpi(p^{\bfj})}{p^{|\bfj|_1}} 
= p^{-2} +O(p^{-2-\eps}),$$
for some $\epsilon>0$.
Assorting all related estimates we obtain for square-free $\del$ that
$$\prod_{\substack{z_0<p\leq z_1\\ p|\del}}
\sum_{\substack{\bfj\in \{0,1\}^n\\ j_{l_1}=j_{l_2}=1}}
\frac{\varpi(p^{\bfj})}{p^{|\bfj|_1}}
\ll
\del^{-2+\eps},$$
and 
\begin{equation*}
\begin{split}
\prod_{\substack{z_0<p\leq z_1\\ p\nmid\del}}
\sum_{\bfj\in \{0,1\}^n}
\frac{\varpi(p^{\bfj})}{p^{|\bfj|_1}}  
&\ll \prod_{z_0< p\leq z_1}\left(1+\frac{n}{p}+O(p^{-1-\eps})\right) \ll (\log z_1)^n.
\end{split}
\end{equation*}
These estimates prove that 
$\calE(l_1,l_2)\ll 
(\log z_1)^n
\sum_{\del>\Del}\del^{-2+\eps}
$,
which is sufficient.
\end{proof} 
For any square-free integer $m$ and 
index $1\leq i \leq n$ define 
\[
\varpi_i(m):=\varpi(1,\ldots,1,m,1,\ldots,1)
,\] where $m$ appears in the $i$-th position.
For $\eps>0$ define the multiplicative function
\[
\phi_\eps(m):=\prod_{\substack{p|m\\p>z_0}}(1+p^{-\eps})
.\]
Note that if assumptions of Corollary~\ref{cor2} hold for $|\bfj|_1=1$
then
there exists $\eps=\eps(\b{f})>0$ such that
$\sigma_p(p^{\boldsymbol{e}_i}|\bfx)=\frac{1}{p}+O(p^{-1-\eps})$.  
Enlarging $z_0$ and replacing $\eps$ by a smaller positive constant
if needed
yields the following result.
\begin{lemma}
\label{lem:general}
Assume that
$\mathfrak{B}(\bff)> 
\max\big\{
R^2 2^{d-1}(d^2-1),
R^2
2^{d-1}
(d-1)
+(R+1)
\big\}
$.
Then there exists $\eps=\eps(\b{f})>0$ such that for all square-free integers $m$,
\[\max_{1\leq i \leq n}\varpi_i(m) \leq \phi_\eps(m).\]
\end{lemma}
Observe 
 that for all
$\bfd \in \N^n$ with $\mu(\bfd)^2=1$ the expression
\[
\frac{\varpi(\bfd)}{\prod_{i=1}^n\varpi_i(d_i)}
\]
is a function of the vector $(\gcd(d_i,d_j))_{1\leq i<j \leq n}$.
To see this, it is enough to consider the case when $\widetilde{\bfd}$
is divisible by a single prime, say $p$. We need to show that 
if 
$\bfh, \bfk \in \{0,1\}^n$
and 
\begin{equation}
\label{eq:minim}
i\neq j \Rightarrow
\min(k_i,k_j)
=
\min(h_i,h_j)
\end{equation}
then 
\begin{equation}
\label{eq:minimeq}
\frac{\varpi(p^{\bfk})}{\prod_{i=1}^n\varpi_i(p^{k_i})}
=
\frac{\varpi(p^\bfh)}{\prod_{i=1}^n\varpi_i(p^{h_i})}
.\end{equation}
Obviously this holds in the case that $\bfk=\bfh$
and we can therefore assume that $\bfk\neq \bfh$.
A little thought reveals that in this case~\eqref{eq:minim}
guarantees that there exist $l,m,i \neq j$ such that 
$(\bfk,\bfh)$ equals one of the following,
\[
(\boldsymbol{e}_l,\boldsymbol{0}),
(\b{0},\boldsymbol{e}_m),
(\boldsymbol{e}_i,\boldsymbol{e}_j)
.\]
For any such instance
we can verify that both sides of~\eqref{eq:minimeq}
equal $1$, hence our claim holds.
We have proved that there exists a function
$
\widehat{g}
:\N^{{n}\choose{2}}\to \R_{\geq 0}$ 
such that 
\[\mu(\bfd)^2=1
\Rightarrow
\varpi
(\bfd)=
\widehat{g}((d_{i,j}))
\prod_{i=1}^n
\varpi_i(d_i)
.\]
The function $\varpi_i(d_i)$
keeps track of the probability that $d_i|x_i$
and the function $\widehat{g}((d_{i,j}))$ 
takes values close to $1$
when the events $d_i|x_i$ 
are independent (in a suitable sense)
but can obtain larger values in general. 

Defining 
\[S((u_{i,j})):=\sum_{\substack{\b{k} \in \N^n\\k_i|P(z_0,z_1)\\(k_i,k_j)=u_{i,j}}}\prod_{i=1}^n\frac{\lambda^+_{k_i}\varpi_i(k_i)}{k_i}\]
enables us to write
\begin{equation}\label{eqnSigstar}
\Sig^*(D_1,z_1)=
\sum_{\substack{u_{i,j}\leq \Delta
\\1\leq i<j \leq n}}\widehat{g}((u_{i,j}))S((u_{i,j}))
.\end{equation}
We may now use 
the expression
$
(\mu\ast 1)((k_i/u_{i,j},k_j/u_{i,j}))
$
to detect the condition $(k_i,k_j)=u_{i,j}$,
thus inferring
\begin{equation}
\label{eq:2btrunc}
S((u_{i,j}))=
\sum_{\substack{(l_{i,j}) \in \N^{{n \choose 2}}\\1\leq i \neq  j \leq n\\ u_{i,j}l_{i,j}|P(z_0,z_1)}}
\hspace{-0,5cm}\mu(\b{l})
\prod_{i=1}^n
\l(\sum_{\substack{k\in \N\\k|P(z_0,z_1)\\ \xi_i|k}}
\frac{\lambda^+_k\varpi_i(k)}{k}
\r)
,\end{equation}
where
\[
\xi_i:=
\text{rad}\l(
\prod_{\substack{1\leq j \leq n \\ j\neq i}}u_{i,j}l_{i,j} 
\r)
\]
and $\text{rad}$ stands for the radical of a positive integer.
Under the assumptions of Lemma~\ref{lem:general} we thus obtain
the following estimate for all square-free integers $\delta$,
\[
\Bigg|
\sum_{\substack{k\in \N\\k|P(z_0,z_1)\\ \delta|k}}
\frac{\lambda^+_k
\varpi_i(k)}{k}
\Bigg|
\leq 
\frac{\phi_\eps(\delta)}{\delta} 
\prod_{z_0<p\leq z_1}(1+p^{-1}+p^{-1-\eps})
\ll \frac{\phi_\eps(\delta)}{\delta}  \log z_1
.\]
Note that the succeeding inequality holds 
for all divisors $m'$ of $m$, 
$$\frac{\phi_\eps(m)}{m} \leq \frac{\phi_\eps(m')}{m'}.$$
Letting $\xi_i^*$ be the radical of $\prod_{j\neq i}l_{i,j}$
and using the last inequalities with
$\delta=m=\xi_i$ and $m'=\xi_i^*$ 
allows us to truncate the sum in~\eqref{eq:2btrunc}
to the range $l_{i,j}\leq \Delta^{B_1}$, 
where $B_1>0$ is a constant that will be chosen in due course.
The contribution of $l_{1,2}>\Delta^{B_1}$ is 
\[
\ll
(\log z_1)^n
\sum_{\substack{l_{i,j}\leq D_1,l_{i,j}|P(z_0,z_1)\\l_{1,2}>\Delta^{B_1}}}
\frac{\mu(\bfl)^2}{\widetilde{\boldsymbol{\xi}^*}}
\prod_{i=1}^n \phi_\eps(\xi_i^*)
,\]
where $D_1$ is the support of $\lambda^+$. 

We may now use the inequality
\[
\phi_\eps(\xi_i^*)\leq \prod_{j\neq i}\phi_\eps(l_{i,j})
\]
to obtain
\[
\prod_{i=1}^n \phi_\eps(\xi_i^*)
\leq \prod_{1\leq i\neq j\leq n}\phi_\eps(l_{i,j})^2
.\]
Hence the last sum is 
\[
\leq
\sum_{\substack{l_{1,2}>\Delta^{B_1}}}
\frac{\mu(\bfl)^2}{\xi_1^*\cdots \xi_n^*}
\prod_{1\leq i\neq j\leq n}\phi_\eps(l_{i,j})^2
.\]
This is really a summation over the variables $l_{1,2},\ldots,l_{n-1,n}$
because each expression $\xi_i^*$ is a function of some of
these variables.
We first perform a summation over $l_{n-1,n}$.
Recalling that 
\[
\xi_i^*=\text{rad}\big(\prod_{j\neq i}l_{i,j}\big)
\]
we see that only $\xi_{n-1}^*$ and $\xi_n^*$ depend on $l_{n-1,n}$,
since they satisfy
\[
\xi_n^*=[l_{n-1,n},\xi_n^{**}],
\xi_{n-1}^*=[l_{n-1,n},\xi_{n-1}^{**}],
\]
where both $\xi_{n-1}^{**}$ and $\xi_{n}^{**}$
are defined as 
$\xi_{n-1}^{*}$ and $\xi_{n}^{*}$
but with the variable $l_{n-1,n}$ missing, i.e.
\[
\xi_{n-1}^{**}:=\text{rad}\big(\prod_{j\neq n-1,n}l_{n-1,j}\big),
\xi_{n}^{**}:=\text{rad}\big(\prod_{j\neq n-1,n}l_{n,j}\big)
.\]
Hence the sum over $l_{n-1,n}$ is 
\[
\sum_{l_{n-1,n}}\frac{\mu(l_{n-1,n} )^2\phi_\eps(l_{n-1,n})^2}{[l_{n-1,n},\xi_{n-1}^{**}][l_{n-1,n},\xi_{n}^{**}]}
,\]
which equals
\[
\frac{1}{\xi_{n-1}^{**}\xi_{n}^{**}}
\sum_{l_{n-1,n}}
\frac{\mu(l_{n-1,n})^2\phi_\eps(l_{n-1,n})^2}{l_{n-1,n}^2}
\gcd(\xi_{n-1}^{**},l_{n-1,n})
\gcd(\xi_{n}^{**},l_{n-1,n})
.\]
The last sum is 
\[\leq 
\prod_{p|\xi_{n-1}^{**}\xi_{n}^{**}}(2+2p^{-\eps}+p^{-2\eps})
\prod_p
(1+p^{-2}(1+p^{-\eps})^2)
\ll \tau(\xi_{n-1}^{**})^A
\tau(\xi_{n}^{**})^A
,\]
where $A=3$.
Of course we can bound any
$\xi_k^{**}$ by the product of all available
variables except $l_{n-1,n}$, i.e.
$\prod_{\{i,j\}\neq \{n-1,n\}}l_{i,j}$,
thus we obtain 
\[
\ll
\frac{1}{\xi_{n-1}^{**}\xi_{n}^{**}}
\prod_{\{i,j\}\neq \{n-1,n\}}\tau(l_{i,j})^{2A}
.\]
The process above is the first step of a finite induction that eliminates all variables $l_{i,j}$,
beginning from $l_{n-1,n}$ and terminating with $l_{1,2}$.
At each step expressions of the form   
\[
\sum_{\substack{l_{1,2}>\Delta^{B_1}}}
\frac{\mu(\bfl)^2}{\xi_1'\cdots \xi_n'}
\Oprod_{1\leq i\neq j\leq n}\tau(l_{i,j})^A
\]
are bounded by 
\[\ll
\sum_{\substack{l_{1,2}>\Delta^{B_1}}}
\frac{\mu(\bfl)^2}{\xi_1^{''}\cdots \xi_n^{''}}
\OOprod_{1\leq i\neq j\leq n}\tau(l_{i,j})^{100 A}
,\]
where the notation $\xi',\Oprod$ means that some of the variables $l_{i,j}$ have been eliminated,
the notation $\xi^{''},\OOprod$ that one further variable has been eliminated
and the constant $A'$ depends at most on $A$ and $\b{f}$.
At the last step of the induction we will arrive at the expression
\[
\sum_{l_{1,2}>\Delta^{B_1}}\frac{\mu(l_{1,2})^2}{l_{1,2}^2}\tau(l_{1,2})^C
,\] where $C=C(\b{f})$.
Obviously this is 
$\ll
\Delta^{-\frac{B_1}{2}}
$. 
The arguments above show that
\begin{equation}
\label{eq:2btrunctrunc}
S((u_{i,j}))=
\sum_{\substack{(l_{i,j}) \in \N^{{n \choose 2}}\\l_{i,j}\leq \Del^{B_1}\\ u_{i,j}l_{i,j}|P(z_0,z_1)}}
\hspace{-0,5cm}\mu(\b{l})
\prod_{i=1}^n
\l(\sum_{\substack{k\in \N\\k|P(z_0,z_1)\\ \xi_i|k}}
\frac{\lambda^+_k\varpi_i(k)}{k}
\r)+O\left((\log z_1)^n \Del^{-\frac{B_1}{2}}\right),\end{equation}
where the implied constant is independent of the $u_{i,j}$.

We now aim to use a consequence of the linear
case
of the
Rosser--Iwaniec sieve
(in fact the linear case was 
settled first by Jurkat and Richert~\cite{MR0202680})
that is given in~\cite[Lem.11]{BF}. 
We shall find it convenient to use the error term appearing in~\cite[Th.1]{MR581917},
this will lead to replace the term 
$
e^{\sqrt{L-s}}(\log D)^{-1/3}
$
in~\cite[Lem.10]{BF}
and~\cite[Lem.11]{BF}
by
\[
e^{\sqrt{L}}
Q(s)(\log D)^{-1/3}
\]
where, as stated in~\cite[Eq.(1.6)]{MR581917}, the function $Q(s)$ satisfies
\[Q(s)<\exp\{-s \log s+s \log \log 3s+O(s)\}
, s\geq 3
.\] 
The constant $L$ in our case will depend at most on the coefficients of $\b{f}$,
which is considered constant throughout our paper-thus we can assume that 
the terms above are 
$\ll_\bff
s^{-s}
(\log D)^{-1/3}
$, with an implied constant depending at most on $\bff$.
Let us choose the set of primes 
$$\calP:=\{p \mbox{ prime }: p>z_0\}.$$
Moreover, we observe that $\varpi_i(k)$ is a multiplicative function for all $1\leq i\leq n$. 
We define the modified multiplicative function $\widetilde{\varpi_i}(k)$ by
\begin{equation*}
\widetilde{\varpi}_i(p):=\left\{\begin{array}{ccc} 
\varpi_i(p)&\mbox{if} & p>z_0\\ 
0 &\mbox{if} & p\leq z_0.\end{array}\right.
\end{equation*}
So far we can only assume that $\varpi_i(p)\leq 1+p^{-\eps},$ whereas in \cite{BF} they work with the stronger statement that 
$\varpi(p)\leq 1+1/(p-1).$ However, we still get the bound present in \cite[Eq.(3.10)]{BF} for a uniform $L$. 
For this we observe that
\begin{equation*}
\begin{split}
\log \prod_{w_1<p\leq w_2}\left(1-\frac{\widetilde{\varpi}_i(p)}{p}\right)^{-1} &\leq \sum_{w_1<p\leq w_2} \log \left(1-\frac{1}{p}-\frac{C}{p^{1+\eps}}\right)^{-1} \\&\leq \sum_{w_1<p\leq w_2}\left(\frac{1}{p}+\frac{C}{p^{1+\eps}} \right)+O( w_1^{-1}) \\
& \leq \log\log w_2-\log\log w_1 +O\left(\frac{1}{\log w_1}\right).
\end{split}
\end{equation*}
by Mertens' theorem. This leads to the bound 
$$\prod_{w_1<p\leq w_2}\left(1-\frac{\widetilde{\varpi}_i(p)}{p}\right)^{-1}\leq \left(\frac{\log w_2}{\log w_1}\right)\left(1+\frac{L}{\log w_1}\right),
$$ for a uniform constant $L=L(\bff)$.  
We can now directly apply~\cite[Lem.11]{BF} to the inner sums appearing in~\eqref{eq:2btrunctrunc}. 
Introduce the constant $s_0$ through $s_0:=(\log D_1)/(\log z_1)$,
which we demand that it fullfills $s_0\geq 3$,
and set 
$$U_i(z_1,\xi_i):=\mu(\xi_i)
\prod_{\substack{p|\xi_i\\p>z_0}}\frac{\varpi_i(p)}{p-\varpi_i(p)}
\prod_{\substack{z_0<p\leq z_1 
}}\left(1-\frac{\varpi(p)}{p}\right).$$
This provides us with 
\[
\sum_{\substack{k\in \N\\k|P(z_0,z_1)\\ \xi_i|k}}
\frac{\lambda^+_k\varpi_i(k)}{k} = U_i(z_1,\xi_i)+O\left(\tau(\xi_i)
s_0^{-s_0}
\right).
\]
Owing to the apparent
bounds
$0\leq \varpi_i(p)< p/2$, valid for $p>z_0$ (as long as $z_0$ is enlarged) 
we deduce that $|U_i(z_1,\xi_i)|\leq 1$ for all $1\leq i\leq n$ and divisors $\xi_i|P(z_0,z_1)$.
We use this approximation in~\eqref{eq:2btrunctrunc}, to obtain
$$S((u_{i,j}))- 
\hspace{-0,5cm}
\sum_{\substack{
l_{i,j}\leq \Del^{B_1}\\ u_{i,j}l_{i,j}|P(z_0,z_1)}}
\hspace{-0,5cm}\mu(\b{l})\prod_{i=1}^n U_i(z_1,\xi_i) 
\ll
(\log z_1)^n \Del^{-\frac{B_1}{2}}+\Del^{B_1\binom{n}{2}+1/100}(s_0^{-s_0}+
s_0^{-s_0}
(\log D_1)^{-\frac{1}{3}}) 
.$$
Assume that the assumptions in Lemma \ref{lem3} are satisfied. Together with equation (\ref{eqnSigstar}) we now obtain
$$\Sig(D_1,z_1)= \Sig^{MT}(D_1,z_1)+ \Sig^{ET}(D_1,z_1),$$
with a main term given by
$$\Sig^{MT}(D_1,z_1)=\sum_{\substack{u_{i,j}\leq \Delta
\\1\leq i<j \leq n}}
\widehat{g}((u_{i,j}))\sum_{\substack{(l_{i,j}) \in \N^{{n \choose 2}}\\l_{i,j}\leq \Del^{B_1}\\ u_{i,j}l_{i,j}|P(z_0,z_1)}}
\hspace{-0,5cm}\mu(\b{l})\prod_{i=1}^n U_i(z_1,\xi_i),$$
and an error term satisfying
$$\Sig^{ET}(D_1,z_1)
\ll
\frac{(\log z_1)^n}{
\Del^{1-\eps}  
}
+ 
\frac{
\Del^{
C+
\binom{n}{2}
-\frac{B_1}{2}}
}
{(\log z_1)^{-n}}+
\Del^{C+(B_1+1)\binom{n}{2}+1/100}
s_0^{-s_0} 
,$$
where $C=C(\bff)>0$ is such that 
\[ 
|\widehat{g}((u_{i,j}))| \ll \max\{u_{i,j}\}^C
.\]
We will assume that such a $C$ exists for the moment, 
this will be proved later in Lemma~\ref{lemg}.
Therefore we may 
choose $B_1>0$ large enough 
so that 
$C+
\binom{n}{2}
-\frac{B_1}{2}<-1$.
We can then obtain
\begin{equation}\label{eqn21}
\Sig(D_1,z_1)-\Sig^{MT}(D_1,z_1)
\ll 
\frac{(\log z_1)^n}{\Del^{1-\epsilon}}
+\Del^c s_0^{-s_0}
,\end{equation}
where $c=c(\bff)>0$. Note that here we implicitly assume that $s_0\geq 3$, thus
 $\frac{\log D_1}{\log z_1}\geq 3$. 
\begin{lemma}
\label{lemg}
Assume that
$
\mathfrak{B}(\bff)> 
\max\{
R^2 2^{d-1}(d^2-1),
R^2
2^{d-1}
(d-1)
+(R+1)
\}
$.
Let $\bfu\in \N^{\binom{n}{2}}$ be
such that $\mu^2(\bfu)=1$ and such that $p|\widetilde{\bfu}$ implies that $p>z_0$. Then, for $z_0$ sufficiently large one has
$$\widehat{g}((u_{i,j}))\ll \left(\prod_{i\neq j}u_{i,j}\right)^{\frac{d\mathfrak{B}(\bff)}{(d-1)2^{d-1}}(d-\frac{1}{R})+R+\eps}.$$
\end{lemma}
\begin{proof}
First we recall that
\begin{equation}\label{eqng}
\widehat{g}((u_{i,j}))\prod_{i=1}^n \varpi_i(u_i)= \varpi(\bfu),
\end{equation}
where we have $u_{i,j}=\gcd(u_i,u_j)$. For bounding $\widehat{g}((u_{i,j}))$ we may make the following assumption: if $p$ is a prime with $p|u_i$ for some $1\leq i\leq n$, then there is a $1\leq j\leq n$, $j\neq i$ such that $p|u_{i,j}$. Otherwise we could replace in \eqref{eqng} the vector $\bfu$ with a vector $\widetilde{\bfu}$ where $\widetilde{u_k}=u_k$ for $k\neq i$ and $\widetilde{u_k}=\frac{u_k}{p}$ for $k=i$. In particular, we may assume that
$$u_i\leq \prod_{j\neq i} u_{ij},$$
for every $1\leq i\leq n$.\par
Next we observe that
$$\prod_{i=1}^n \varpi_i(u_i) = \prod_{i=1}^n \prod_{p|u_i} \varpi_i(p).$$
We recall the identity
$\varpi_i(p)=p\sig(p^{\bfe_i}|\bfx)\sig_p^{-1}$.
By Corollary \ref{cor2} we have
$$\sig(p^{\bfe_i}|\bfx)=\frac{1}{p}+O(p^{-1-\eps}).$$
Therefore we obtain
$$\prod_{i=1}^n \varpi_i(u_i)=\prod_{i=1}^n \prod_{p|u_i}(1+O(p^{-1-\eps}))^{-1}(1+O(p^{-\eps})),$$
and
$\prod_{i=1}^n \varpi_i(u_i)^{-1}\ll_\mu (u_1\cdots u_n)^{\mu},$
for any $\mu>0$. 
By Corollary \ref{corome} we have
$$\varpi(p^{\bfj})\ll p^{\frac{d\mathfrak{B}(\bff)}{(d-1)2^{d-1}}
(d-\frac{1}{R})
+R }.$$
Injecting
these bounds into~\eqref{eqng} yields
\[
\widehat{g}((u_{i,j}))\ll \Big(
\prod_{p|\widetilde{\bfu}}p\Big)^{\frac{d\mathfrak{B}(\bff)}{(d-1)2^{d-1}}(d-\frac{1}{R})+R+\eps}
\ll \Big(\prod_{i\neq j}u_{i,j}\Big)^{\frac{d\mathfrak{B}(\bff)}{(d-1)2^{d-1}}(d-\frac{1}{R})+R+\eps},
\]
thus concluding the proof.
\end{proof}

As in~\cite[p.90]{BF},
we now observe that $\Sig^{MT}(D_1,z_1)$ is independent of $D_1$. We set 
$$D_2:= \max (D_1,3^{z_1})$$
and with equation (\ref{eqn21}) applied to $D_2$ instead of $D_1$, we obtain that
\begin{equation*}
\Sig(D_1,z_1)-\Sig(D_2,z_1)
\ll 
\frac{(\log z_1)^n}{\Del^{1-\epsilon}}
+\Del^c s_0^{-s_0}
.\end{equation*}
For this choice of $D_2$ we have $\lam^+_d=\mu(d)$ for $d|P(z_0,z_1)$ (note that with the change of $D_1$ to $D_2$ also the sieve weights $\lam$ change). Hence we can compute $\Sig(D_2,z_1)$ as
$$\Sig(D_2,z_1)=\sum_{\substack{\b{d} \in \N^n
\\ d_i |P(z_0,z_1)}}
\frac{\varpi(\bfd)}{\widetilde{\bfd}}\prod_{i=1}^n\mu(d_i)=\prod_{z_0<p\leq z_1}\left(1-\frac{g(p)}{p}\right),$$
with $g(p)$ defined as in \eqref{def:ome}.
Injecting our estimates for $\Sigma(D_1,z_1)$ into~\eqref{eq:bigsigma}
yields the upper bound in the next result.
\begin{proposition}
\label{prop:presiev}
Assuming
$
l_i|P(z_1,z),
|\b{l}|D_1\log B
\leq B^{1/\rho}
$
and that $\mathfrak{B}(\b{f})$ exceeds
\[ 
\max\big\{
2^{d-1}(d-1)
R(R+1),
2^{d-1}
(d-1)R^2
+(R+1)
(\Upsilon+1),
2^{d-1}
(d^2-1)R^2
\big\}
\]
we have
\begin{equation*}
\begin{split} 
G(B,z_1;\b{l})
&
=
X_\bfl \prod_{z_0<p\leq z_1}\left(1-\frac{g(p)}{p}\right)
 +O\left( 
\varpi(\bfl)
\frac{B^{n-Rd}}{{\widetilde{\b{l}}}} 
\left(\frac{(\log z_1)^n}{\Del^{1-\epsilon}}
+\Del^c s_0^{-s_0}\right)\right)\\& +
O\Bigg( 
\frac{B^{n+\epsilon}}{
\widetilde{\b{l}}
}
\sum_{\substack{|\bfk|
\leq D_1\\
p|\widetilde{\bfk} \Rightarrow z_0<p \leq z_1}}
\mu(\b{k})^2
\frac{E(B;
(k_1l_1,\ldots,k_nl_n)
)}
{\widetilde{\b{k}}} 
\Bigg).
\end{split}
\end{equation*}
\end{proposition}
The lower bound can be procured
upon writing
\[
G(B,z_1;\b{l})
=
\sum_{\substack{
\b{x} \in \c{A}\\ \bfl|\bfx}}
\left(
\prod_{i=1}^n
w\bigg( \frac{x_i}{B} -\frac{\zeta_i}{2|\boldsymbol{\zeta}|} \bigg) \right)
\left(
\prod_{i=1}^n
(1\ast \mu)
(\gcd(P(z_0,z_1),x_i))
\right)
\]
and using Lemma~\ref{lem:multi} to obtain
\[
G(B,z_1;\b{l})
\geq 
\sum_{i=0}^n 
c_i M_i 
,\]
where for 
$1\leq i \leq n$
we define  $c_i:=1$ and 
\[ 
M_i:=
\sum_{\substack{ 
k_i |P(z_0,z_1)}}
\Bigg(
\lambda_{k_i}^-
\prod_{j\neq i}\lambda_{k_j}^+
\Bigg)
N_w(B;(k_1l_1,\ldots,k_nl_n)) 
,\] in addition to 
$c_0:=-(n-1)$ 
and 
\[ 
M_0:=
\sum_{\substack{ 
k_i |P(z_0,z_1)}}
\Bigg(\prod_{i=1}^n\lambda_{k_i}^+\Bigg)
N_w(B;(k_1l_1,\ldots,k_nl_n))
.\]  
The treatment of each individual $M_i, (i\neq 0)$, is identical to the treatment of $M_0$ earlier
in this section.
The only difference arises at the last step (the calculation of the Euler products in the main term).
Here the coefficients $c_i$ satisfy  $\sum_{0\leq i \leq n}c_i=1$,
thus completing the proof of Proposition~\ref{prop:presiev}.
\section{Proof of Theorems~\ref{thm:mainvector} and~\ref{thm:levelsaturation}
}
\label{s:mainthm}
Recall the
definition of the set $\c{A}$ in~\eqref{def:cala}.
Our aim is to find a large function
$z=z(B)\leq B$
such that  
\[
S(B,z)
:=
\#\big\{
\b{x}\in \c{A}: 
|\bfx|\leq B,
p|x_1\cdots x_n \Rightarrow p>z
\big\}
\gg
\frac{B^{n-Rd}}{(\log B)^n}
.\]
By~\eqref{eq:bwvfminor} we have 
\begin{equation} 
\label{eq:bwv1043largo}
S(B,z) 
\geq w_0^{-n}S_{\boldsymbol{\zeta}}(B,z)
,\end{equation} 
where 
\[
S_{\boldsymbol{\zeta}}(B,z)
:=
\sum_{\substack{
\b{x}\in \c{A}
\\
p|x_1\cdots x_n \Rightarrow p>z
}}
\prod_{i=1}^n
w\bigg( \frac{x_i}{B} -\frac{\zeta_i}{2|\boldsymbol{\zeta}|} \bigg)
.\]
One may now
write the sum over $\bfx$ as 
\[
\sum_{\substack{
\b{x}\in \c{A}
}}
\left(
\prod_{i=1}^n
w\bigg( \frac{x_i}{B} -\frac{\zeta_i}{2|\boldsymbol{\zeta}|} \bigg)
\right)
\left(
\prod_{i=1}^n
(1\ast \mu)
(\gcd(P(z),x_i))
\right)
.\]
For a parameter $D$ let $\lambda^{\pm}$ be a
sieve sequence
supported in $[1,D]$. 
Letting 
\[
\beta(\b{l})
:=
\sum_{i=1}^n
\lambda_{l_i}^-
\Big(
\prod_{\substack{1\leq j \leq n\\j \neq i }}
  \lambda_{l_j}^+
\Big)
-(n-1)
\prod_{i=1}^n
\lambda_{l_i}^+,
\
\b{l} \in \N^n,
\]
alluding to Lemma~\ref{lem:multi} and recalling~\eqref{def:www}
allows us to infer that 
for any $z>z_1$
we have 
\[
S_{\boldsymbol{\zeta}}(B,z)
\geq 
\sum_{\substack{\b{l} \in \N^n\\ l_i \mid P(z_1,z)}}
\!\!
\beta(\b{l})
G(B,z_1;\b{l})
.\]
Define the entities
\[\Sigma(D,z_1,z):=\sum_{\substack{\b{l} \mid P(z_1,z)}}\beta(\b{l})\frac{\varpi(\bfl)}{\widetilde{\bfl}},\]
\[
B_1:= 
\sum_{\substack{\bfl \mid P(z_1,z)}}\frac{\varpi(\bfl) }{\widetilde{\bfl}}
\ \ \text{and} \ \
B_2:= B^{dR+\epsilon}
\sum_{\substack{|\bfl|
\leq D\\
\bfl| P(z_1,z)
}}
\frac{1}{\widetilde{\b{l}}}
\sum_{\substack{|\bfk|\leq D_1\\\bfk| P(z_0,z_1)}}  
\frac{E(B;
(k_1l_1,\ldots,k_nl_n)
)}
{\widetilde{\b{k}}} 
.\]
Proposition~\ref{prop:presiev}
now leads to
\begin{equation}
\label{eq:ineq}
\frac{S_{\boldsymbol{\zeta}}(B,z)
}{\mathfrak{S}(\bff) \mathcal{J}_w(\bff,W) B^{n-Rd}}
\geq 
\Sigma(D,z_1,z) \hspace{-0,2cm}
\prod_{z_0<p\leq z_1}
\hspace{-0,2cm}
\left(1-\frac{g(p)}{p}\right)
\!+
O\!\l(\!\left(\!\frac{(\log z_1)^n}{\Del^{1-\epsilon}}+\frac{\Del^c}{s_0^{s_0}}\!\right)
\!
B_1\!+\!B_2\!\r)\!
.\end{equation}
Letting $m_i:=k_il_i$ and taking advantage of the coprimality of $k_i,l_i$ 
shows that 
\[B_2
\leq 
B^{dR+\epsilon}
\sum_{\substack{|\b{m}|\leq DD_1\\
\b{m}| P(z_0,z)}}  
\frac{E(B;\b{m})}
{\widetilde{\b{m}}} 
.\]
Recalling the definition of the matrix
$\boldsymbol{\epsilon}$ given in~\eqref{def:epss},
shows that,
under the condition
\[
D D_1\leq \frac{B^{1/\rho}}{\log B}
,\]
the sum over $\b{m}$ is
\[
\ll
\sum_{i=1}^3
B^{-\epsilon_{i,1}}
\sum_{1\leq m_1\leq D D_1}
m_1^{\epsilon_{i,2}-1}
\sum_{1\leq m_2\leq m_1}
m_2^{-1}
\ldots
\sum_{1\leq m_{n-1}\leq m_{n-2}}
m_{n-1}^{-1}
\sum_{1\leq m_n\leq m_{n-1}}
m_n^{\epsilon_{i,3}-1}
.\]
Since each $\epsilon_{i,j}$ is non-negative 
we can use the estimate 
$\sum_{1\leq m \leq z}m^{\lambda-1}\ll_\lambda z^{\lambda} \log z$,
valid for each fixed $\lambda \geq 0$ 
to deduce that for every $\epsilon>0$ one has 
\[
B_2
\ll 
B^{dR+\epsilon}
\sum_{i=1}^3
B^{-\epsilon_{i,1}}
(D D_1)^{\epsilon_{i,2}+\epsilon_{i,3}}
.
\]
Our remaining task will be to give a lower bound for $\Sigma$ and an upper bound for $B_1$.
We begin by studying the contribution to $\Sigma(D,z_1,z)$
of vectors $\bfl$ with $\delta:=\gcd(l_{i_1},l_{i_2})\neq 1$; this task is similar to the one in Lemma~\ref{lem3}
and we adapt its assumptions in what follows.
Each such $\delta$ is a product of primes $p>z_1$,
therefore this contribution is 
\[
\ll
\sum_{\delta>z_1} \mu(\delta)^2
\sum_{\substack{\bfl |P(z_1,z)\\ \del | l_{i_1}, \del |l_{i_2}}} 
\frac{\varpi(\bfl)}{\widetilde{\bfl}}
.\]
As in the proof of Lemma~\ref{lem3} we find that this is 
\[
\ll \l(\frac{\log z}{\log z_1}\r)^n\sum_{\delta>z_1}\delta^{-2+\eps} 
\ll  z_1^{-1+\eps} (\log z)^n
.\]
Note that if $l_i,l_j$ are coprime for all $i\neq j$ guarantees that 
$\varpi(\bfl)=\prod_{i=1}^n \varpi_i(l_i)$.
This gives
\[\Sigma(D,z_1,z)=
\sum_{\substack{\b{l} \mid P(z_1,z)\\ i\neq j \Rightarrow \gcd(l_i,l_j)=1}}
\frac{\beta(\b{l})}{\widetilde{\bfl}}\prod_{i=1}^n \varpi_i(l_i)
+O(z_1^{-1+\eps} (\log z)^n).\]
The same argument can also be used to establish  
\[\sum_{\substack{\b{l} \mid P(z_1,z)}}
\frac{\beta(\b{l})}{\widetilde{\bfl}}\prod_{i=1}^n \varpi_i(l_i)
=
\sum_{\substack{\b{l} \mid P(z_1,z)\\ i\neq j \Rightarrow \gcd(l_i,l_j)=1}}
\hspace{-0,5cm}
\frac{\beta(\b{l})}{\widetilde{\bfl}}\prod_{i=1}^n \varpi_i(l_i)
+O(z_1^{-1+\eps} (\log z)^n).\]
Letting
\[\Psi^\pm_i:=\sum_{l|P(z_1,z)}\lambda^\pm_l
\frac{\varpi_i(l)}{l}\]
shows that the sum on the left equals
\[
\Psi:=
\sum_{i=1}^n
\Bigg(\Psi_i^- \prod_{\substack{1\leq j \leq n\\j \neq i }}\Psi_j^+\Bigg)
-(n-1)
\prod_{i=1}^n
\Psi_i^+
,\]
thus providing
\[
\Sigma(D,z_1,z)=\Psi
+O(z_1^{-1+\eps} (\log z)^n).\]

Under the assumptions of Lemma~\ref{lem3} we can similarly show that the contribution of 
$\bfl$ with $\gcd(l_{i_1},l_{i_2})\neq 1$ to $B_1$ is 
\[
\ll 
\sum_{\delta>z_1}\mu(\delta)^2
\sum_{\substack{
\bfl \mid P(z_1,z)\\ \delta|l_{i-1}, \delta| l_{i_2}}}
\frac{\varpi(\bfl) }{\widetilde{\bfl}}
\ll 
B^{n-Rd}
(z_1^{-1+\eps} (\log z)^n)
.\]
Therefore 
\[
B_1
\ll 
 z_1^{-1+\eps} (\log z)^n
+\sum_{\substack{\bfl \mid P(z_1,z)\\ i\neq j \Rightarrow \gcd(l_i,l_j)=1}}
\prod_{i=1}^n\frac{\varpi_i(l_i)}{l_i}
\]
and the last sum is 
\[
\leq \prod_{i=1}^n 
\sum_{l|P(z_1,z)}\frac{\varpi_i(l)}{l}
\leq 
\prod_{i=1}^n
\prod_{z_1<p\leq z}
\l(1+\frac{1}{p}+O(p^{-1-\eps})\r)
\ll (\log z)^n,\]
hence
$
B_1 \ll
(\log z)^n
$.
We therefore find via~\eqref{eq:ineq} the following lower bound
\begin{equation*}
\begin{split} 
\frac{
S_{\boldsymbol{\zeta}}(B,z)
}{\mathfrak{S}(\bff) \mathcal{J}_w(\bff,W) B^{n-Rd}
}
&
\geq
\Psi
\prod_{z_0<p\leq z_1}\left(1-\frac{g(p)}{p}\right)
+O\left(B^{dR+\epsilon}
\sum_{i=1}^3
B^{-\epsilon_{i,1}}(D D_1)^{\epsilon_{i,2}+\epsilon_{i,3}} \right)
\\& +O\left( 
\frac{ (\log z)^n}{z_1^{1-\eps}(\log z_1)^{n}}
+\left(\frac{(\log z_1)^n}{\Del^{1-\epsilon}}+\frac{\Del^c}{s_0^{s_0}}\right)(\log z)^n
\right),
\end{split}
\end{equation*} 
where a use of 
\[\prod_{z_0<p\leq z_1}\left(1-\frac{g(p)}{p}\right)
\ll (\log z_1)^{-n}
\]
has been made; this can be inferred from the estimate 
$g(p)=\frac{n}{p}+O(p^{-1-\eps})$.
Let us now fix any $\theta>0$ which satisfies 
$\theta<
\theta'$,
where $\theta'$ was defined in~\eqref{def:the}.
Then there exists a small positive $\theta_1$ such that if 
$D:=B^\theta$
and
$D_1:=B^{\theta_1}$
then 
\[
B^{dR+\epsilon}
\sum_{i=1}^3
B^{-\epsilon_{i,1}}
(D D_1)^{\epsilon_{i,2}+\epsilon_{i,3}} 
\ll
B^{-\delta}
,\]
for some $\delta>0$ that is independent of $B$.
Choosing $\Delta=z_1=(\log B)^{2n+1}$
shows that 
\[
s_0=\frac{\log D_1}{\log z_1}
=\frac{\theta_1 \log B}{(2n+1) \log \log B}
\to
\infty
,\]
hence one can verify that 
\[
\frac{ (\log z)^n}{z_1^{1-\eps}(\log z_1)^{n}}
+\left(\frac{(\log z_1)^n}{\Del^{1-\epsilon}}+\frac{\Del^c}{s_0^{s_0}}\right)(\log z)^n
\ll
\frac{1}{(\log B)^n \log \log B}
\]
and
\[
\frac{S_{\boldsymbol{\zeta}}(B,z)
}{\mathfrak{S}(\bff) \mathcal{J}_w(\bff,W) B^{n-Rd}
}\geq
\Psi
\prod_{z_0<p\leq z_1}\l(1-\frac{g(p)}{p}\r)+
O\l((\log B)^{-n}(\log \log B)^{-1}\r)
.\]
The last product is $\gg (\log z_1)^{-n}$,
thus it remains to show that 
$
\Psi
\gg (\log z_1/\log z)^{n}
$.
Let
$s:=\log D/\log z$
and assume that $s>2$.
Using the inequalities stated in~\cite[Lem.10]{BF}
one deduces that
when 
$s=O_n(1)$
with an implied constant depending at most on $n$,
then 
\[
\Psi\geq (\Psi_n(s)+O_n((\log B)^{-1/3}))
\prod_{i=1}^n
\prod_{z_1<p\leq z}
\Big(1-\frac{\varpi_i(p)}{p}\Big)
,\] 
where
$
\Psi_n(s):=
n f(s)-(n-1) F(s)^n
$.
Here $f(s)$ and $F(s)$
denote the well-known 
functions associated to
the linear sieve,
their definition can be found
in~\cite[Eq.(12.1),Eq.(12.2)]{MR2647984}, for example.
Further information on $f$ and $F$ 
is located in~\cite[\S 11,\S 12]{MR2647984}.
In light of the 
last lower bound for $\Psi$,
it is
sufficient to
find the smallest possible 
value for $s$
such that $\Psi_n(s)>0$.
This is equivalent to
\begin{equation}
\label{eq:differential}
\frac{\ F(s)^n}{f(s)}<1+\frac{1}{n-1}
.\end{equation}
It is a standard fact that when $s>2$ then
$0<f(s)\leq 1 \leq F(s)$.
Therefore
if $s$ remains constant and
independent of $n$ then one cannot prove~\eqref{eq:differential} for large $n$,
this forces us to take $s$ as a function of $n$ that tends to infinity.
At this point we have to employ
asymptotic approximations
for 
$f(s)$ and $F(s)$, these can be found
in~\cite[Eq.(11.134)]{MR2647984}.
They are given by
\[F(s),f(s)=1\pm
\exp\Big\{-s \log s-s\log \log s+s+O\Big(\frac{s \log \log s}{\log s}\Big)\Big
\}
\]
and one sees that if  
$
s\geq 
3(\log n)(\log \log n)^{-1}
$
then 
\[
s\log s
\geq
3 \log n+\frac{3 (\log 3) (\log n)}{\log \log n}
-\frac{3 (\log n) (\log \log \log n)}{\log \log n}
.\]
Therefore, for all large enough $n$, say $n\geq n_0$ for some positive 
absolute constant $n_0$, we obtain
\[
\frac{\ F(s)^n}{f(s)}<\Big(
1+\frac{1}{n^{5/2}}
\Big)^{n+1}
\]
and the inequality $1+n^{-5/2}<(1+(n-1)^{-1})^{1/(n+1)}$, valid for all $n\geq 2$,
makes~\eqref{eq:differential} available. 
In the case that $1\leq n<n_0$ one can immediately infer
from the approximations to $F(s)$ and $f(s)$
that if $s\to+\infty$ then~\eqref{eq:differential}
is automatically satisfied.
This gives a constant $\sigma_0$ that depends at most on $n_0$ (and is therefore absolute)
such that~\eqref{eq:differential}
is valid whenever 
 $s\geq \sigma_0$.  
Hence there exists a positive absolute constant $\sigma_0$
such that if 
\[
s\geq \frac{3 \log n}{\log \log n}+\sigma_0
\]
then, alluding to~\eqref{eq:bwv1043largo}, 
Theorem~\ref{thm:mainvector}
holds
with any constant $c_0>3+\sigma$
and with 
$P^-(x_1\cdots x_n)$
exceeding the sieving parameter
$
z=D^{1/s}=B^{\theta/s}
$.

 \begin{proof}[\textbf{Proof of 
Theorem~\ref{thm:levelsaturation}
}\!\!\!\!
] 
The arguments in the present section have so far 
proved that 
\[
S_{\boldsymbol{\zeta}}(B,B^{\theta/s}) \gg B^{n-Rd} (\log B)^{-n} 
.\]
This is sufficient for 
Theorem~\ref{thm:levelsaturation}
because
to show that a subset of $V_{\b{f}}(\Q)$
is Zariski dense in an absolutely irreducible variety $V_\b{f}$,
it is sufficient to choose an arbitrary neighbourhood
in the real analytic topology of a non-singular point $\boldsymbol{\zeta}\in V_\b{f}(\R)$
and show that any real point in the neighbourhood (on the variety) can be approximated by a rational point. 
In our case the 
neighbourhood is given by $B \c{B}_\eta$ (where $\c{B}_\eta$ was defined in~\eqref{def:bachbwv1043allegro}).
\end{proof}

\bibliographystyle{amsalpha} 
\bibliography{sb}
\end{document}